\documentclass[11pt,reqno]{amsart}

\usepackage{amsmath,amssymb,amscd,graphicx, amsthm, fullpage,
  enumerate, hyperref} 

\usepackage{mathrsfs}
\usepackage{mathtools}
\usepackage[all]{xy}

\usepackage{color}
\definecolor{darkred}{rgb}{1,0,0} 
\definecolor{darkgreen}{rgb}{0,0.8,0}
\definecolor{darkblue}{rgb}{0,0,1}

\hypersetup{colorlinks,
linkcolor=darkblue,
filecolor=darkgreen,
urlcolor=darkred,
citecolor=darkgreen}

\numberwithin{equation}{section}

\theoremstyle{definition}
\newtheorem{theorem}{Theorem}
\numberwithin{theorem}{section}
\newtheorem{proposition}[theorem]{Proposition}
\newtheorem{lemma}[theorem]{Lemma}
\newtheorem{defprop}[theorem]{Definition/Proposition}
\newtheorem{corollary}[theorem]{Corollary}

\newtheorem{definition}[theorem]{Definition}
\newtheorem{remark}[theorem]{Remark}
\newtheorem{example}[theorem]{Example}
\newtheorem{notation}[theorem]{Notation}
\newtheorem{construction}[theorem]{Construction}

\theoremstyle{remark}

\newcommand{\inv}{^{-1}} 

\newcommand{\toto}{\rightrightarrows}

\newcommand{\pithin}{\Pi^{\mathrm{thin}}}
\newcommand{\thinfg}{\pi^{\mathrm{thin}}}
\newcommand{\fund}{{\pi_1^\mathrm{thin} } }

\newcommand{\trans}{\mathrm{Trans}_G}
\newcommand{\transu}{\underline{\mathsf{Trans}}_G}

\newcommand{\Aut}{\mathrm{Aut}} 
\newcommand{\gtors}{{G\text{-}\mathsf{tors}}}
\newcommand{\hol}{\mathsf{hol}}
\newcommand{\tp}{\mathsf{tp}}
\newcommand{\rep}{\mathsf{rep}}
\newcommand{\assoc}{\mathsf{assoc}}

\newcommand{\Hom}{\mathsf{Hom}}
\newcommand{\rank}{\mathsf{rank}}

\newcommand{\Man}{\mathsf{{Man}}}
\newcommand{\Difgpd}{\mathsf{DiffGpd}}

\newcommand{\gpd}{\mathsf{{Groupoid}}}

\newcommand{\DS}{\mathsf{Diffeol}}
\newcommand{\Open}{\mathsf{Open}}
\newcommand{\Set}{\mathsf{Set}}
\newcommand{\Bi}{\mathsf{Bi}}

\newcommand{\calG}{{\mathcal G}}
\newcommand{\calH}{{\mathcal H}}
\newcommand{\calP}{{\mathcal P}}

\newcommand{\CR}{{\mathcal R}}
\newcommand{\CS}{{\mathcal S}}
\newcommand{\calU}{{\mathcal U}}
\newcommand{\calX}{{\mathcal X}}
\newcommand{\calY}{{\mathcal Y}}

\newcommand{\scrP}{{\mathscr P}}

\newcommand{\fg}{\mathfrak g}
\newcommand{\fs}{\mathfrak s}
\newcommand{\ft}{\mathfrak t}
\newcommand{\fm}{\mathfrak m}
\newcommand{\fri}{\mathfrak i}
\newcommand{\fu}{\mathfrak u}

\newcommand{\R}{\mathbb{R}}

\title{Parallel Transport on Principal Bundles over 
  Stacks}

\author{Brian Collier} 
\author{Eugene Lerman}
\author{Seth Wolbert} 

\address{Department of Mathematics, University of Illinois, Urbana,
  IL 61801}

\thanks{
B.C. and S.W.  acknowledge support from National Science Foundation grant DMS 08-38434 ``EMSW21-MCTP: Research Experience for Graduate Students" 
}

\begin{document}

\begin{abstract}
  In this paper we introduce a notion of parallel transport for
  principal bundles with connections over differentiable stacks.  We
  show that principal bundles with connections over stacks can be
  recovered from their parallel transport thereby extending the
  results of Barrett, Caetano and Picken, and Schreiber and Waldorf
  from manifolds to stacks.  

  In the process of proving our main result we simplify Schreiber and
  Waldorf's definition of a transport functor for principal bundles
  with connections over manifolds and provide a more direct proof of 
the correspondence between
  principal bundles with connections and transport functors.
\end{abstract}

\maketitle

\tableofcontents

\section{Introduction}  
Recall that a choice of a connection 1-form $A\in \Omega^1(P,\fg)^G$
on a principal $G$-bundle $P$ over a (connected) manifold $M$ and a 
choice of a base point $x\in M$ gives rise to holonomy map
\[
\Omega(M,x) \to \Aut(\textrm{fiber of $P$
  at }x) \simeq G,
\]
where $\Omega(M,x)$ is the set of smooth loops at $x $ in $M$.  The
holonomy map uniquely determines the connection $A$ and, in fact, the
bundle $P$ itself \cite{K}.  If two loops in $\Omega(M,x)$ differ by a
homotopy that sweeps no area, a so-called ``thin homotopy,'' then
their holonomies are the same.  Therefore the holonomy map descends to
a well defined map on the quotient
\[
\calH:\Omega(M,x)/_{\sim} \to G,
\]
where $\sim$ means ``identify thinly homotopic loops.''  The quotient
$\fund(M,x):= \Omega(M,x)/_{\sim}$ is a group and $\calH$ is a
homomorphism. Moreover $\fund(M,x)$ has a smooth structure --- it is a
diffeological group (see Appendix~\ref{app:1} and
Remark~\ref{cor:4.26t}) --- and $\calH$ is smooth. Barrett \cite{B},
motivated by questions coming from general relativity and Yang-Mills
theory, proved that a homomorphism
\[
T: \fund(M,x) \to G
\]
is defined by parallel transport on some principal $G$-bundle with
connection if and only if $T$ is smooth.  More precisely, he proved
that assigning parallel transport homomorphism to a principal bundle
with a connection induces a bijection of sets:
\begin{multline*}
(\textrm{principal bundles with connections over }M)/\mathrm{isomorphisms}\\
\leftrightarrow \qquad (\textrm{smooth homomorphisms } \fund(M,x) \to
G)/\mathrm{conjugation}.
\end{multline*}
Barrett's proofs were simplified by Caetano and Picken \cite{CP}. Wood
\cite{Wood} reformulated Barrett's theorem in terms of the groupoids
of paths in $M$; this obviates the need to choose a base
point. Schreiber and Waldorf \cite{SW} categorified Wood's version of
Barrett's theorem. They showed that assigning holonomy to a bundle
defines an equivalence of categories
\[
 \hol_M: B^\nabla G(M) \to \Hom_{C^\infty}(\pithin (M), G\textrm{-tor}).
\]
Here, and in the rest of the paper, $B^\nabla G(M)$ denotes the
category of principal $G$-bundles with connections over a manifold
$M$, $\pithin(M)$ is the thin fundamental groupoid of $M$ (see
Definition/Proposition~\ref{defprop:thingpd}), $G$-tor is the category
of $G$-torsors (Definition~\ref{def:catG-tors}) and
$\Hom_{C^\infty}(\pithin (M), G\textrm{-tor})$ denotes a category of
functors that are smooth in an appropriate sense.  Schreiber and
Waldorf's definition of $\Hom_{C^\infty}(\pithin (M), G-\textrm{tor})$
is fairly involved and the proof that $\hol_M$ is an equivalence of
categories is indirect.  Nor is it clear if the equivalence $\hol_M$
is natural in the manifold $M$.

In this paper we propose a simple definition of what it means for a
functor $T:\pithin (M)\to G\textrm{-tor}$ to be smooth.  We provide a
sanity check by showing that the parallel transport functor $\hol_M
(P,A)$ defined by a connection $A$ on a principal $G$-bundle $P\to M$
is smooth in the sense of this paper. We then prove that for a
manifold $M$ the functor $\hol_M$ is an equivalence of categories
(Theorem~\ref{thm:sub_main}).  This part of the paper does not require
any knowledge of stacks.

In the second part of the paper we assume that the reader is familiar
with stacks over the site $\Man$ of manifolds. The standard references
are Behrend and Xu \cite{BX}, Heinloth \cite{H} and Metzler
\cite{Metzler}. 

We first prove that the assignment $M\mapsto \trans(M)$ extends to a
contravariant functor $\trans:\Man^{op}\to \gpd$
(Lemma~\ref{lem:5.1}).  By Grothendieck construction this presheaf
defines a category fibered in groupoids (CFG) $\transu\to \Man$.  The
collection of functors 
\[
\{ \hol_M: B^\nabla G(M)\to \trans(M)\}_{M\in \Man}
\]
extends to a 1-morphism of CFGs $ \hol: B^\nabla G \to \transu$
(Lemma~\ref{lem:5.166}).  Since each functor $\hol_M$ is an
equivalence of categories the functor $\hol$ is an equivalence
(Theorem~\ref{thm:main}).  Consequently, since $B^\nabla G$ is a
stack, so is $\transu$ (Corollary~\ref{cor:5.5}).  Together the two
results imply one of the main result of the paper: 
\[
\hol: B^\nabla G \to
\transu 
\]
is an isomorphism of
stacks.  In section~\ref{sec:6} we work out some consequence of
Theorem~\ref{thm:main} for principal bundles with connections over
stacks.  We start by recalling a definition of a principal $G$-bundle
over a stack $\calX$: it is a functor $P:\calX \to BG$, where $BG$
denotes the stack of principal $G$ bundles.  By analogy we introduce
the notion of a principal bundle with connection and of a transport
functor over a (not necessarily geometric) stack $\calX$.  As an
immediate consequence of Theorem~\ref{thm:main} we obtain that for
each stack $\calX$ the functor $\hol$ induces an equivalence of
categories between the categories of principal bundles with
connections over $\calX$ and transport functors over $\calX$
(Theorem~\ref{thm:6.4}).  We then recall that for a CFG $\calX\to
\Man$ and a Lie groupoid $\Gamma$ there is the category
$\calX(\Gamma)$ of cocycles with values in $\calX$. We discuss the
fact that the cocycle category $\calX(\Gamma)$ is equivalent to the
functor category $[[\Gamma_0/\Gamma_1], \calX]$
(Proposition~\ref{prop:Noohi}). Here and elsewhere in the paper
$[\Gamma_0/\Gamma_1]$ denotes the stack quotient of the Lie groupoid
$\Gamma$.  We end Part~2 of the paper by reformulating
Theorem~\ref{thm:6.4} in terms of the cocycle categories: for any Lie
groupoid $\Gamma$ the isomorphism of stack $\hol$ induces an
equivalence $\hol_\Gamma: B^\nabla G(\Gamma)\to \transu(\Gamma)$ of
the cocycle categories (Theorem~\ref{thm:6.7}).

The paper has two appendices. In Appendix~\ref{app:1} we review
the definition of a diffeological space both from a traditional point of view
and as a concrete sheaf of sets.  We prove the folklore result that the thin
fundamental groupoid $\pithin(M)$ of a manifold $M$ is a diffeological
groupoid.  We also prove two technical results that are needed
elsewhere in the paper.  We show that the target map $\ft$ of the
thin fundamental groupoid has local sections
(Lemma~\ref{lem:loc_sections}).  We prove that the assignment
$M\mapsto \pithin(M)$ extends to functor $\pithin:\Man\to \Difgpd $
from the category of manifolds to the category $ \Difgpd$ of
diffeological groupoids.  In Appendix~\ref{app:2} we prove that for
any Lie groupoid $\Gamma$ an equivalence of CFGs $F:\calX\to \calY$
induces an equivalence $F_\Gamma: \calX(\Gamma)\to \calY(\Gamma)$ of the
corresponding cocycle categories (Proposition~\ref{prop:last}).

\part{Parallel transport for bundles over manifolds}

\section{Thin Homotopy and the thin fundamental groupoid}  
\label{sec:2}

In this section following Schreiber and Waldorf we define the thin
fundamental groupoid $\pithin(M)$ of a manifold $M$. Nothing in this
section is new.  Our purpose for presenting this material is to keep
the paper self-contained and to fix notation. To start we recall the
notion of a path with sitting instances of Caetano and Picken
\cite{CP}.

\begin{definition}[A path with sitting instances] 
\label{def:sit}
Let $[a,b] \subset \R$ be a closed interval and $M$ a manifold. A
smooth map $\gamma: [a,b] \to M$ is a {\sf path with sitting
  instances} if $\gamma$ is constant on neighborhoods of $a$ and $b$.
\end{definition}
\begin{remark}\label{rmrk:5.3333}
It will be useful for us to fix a smooth non-decreasing map 
\[
\beta:[0,1] \to [0,1]
\]
which takes the value $0$ on all points sufficiently close to $0$ and
the value $1$ for all points sufficiently close to $1 $.  Given any
path $\gamma:[0,1]\to M$, $\gamma\circ \beta$ is a path with sitting
instances.  Therefore, up to ``reparameterization,'' all paths on $M$
are paths with sitting instances.\footnote{The map $\beta$ is not a
  reparametrization in the strict sense of the word since it is not
  invertible.}
\end{remark}

\begin{notation}
  \label{calP(M)} We denote the set of all paths with
  sitting instances from the interval $[0,1]$ to a manifold $M$ by
  $\calP(M)$:
\[
\calP(M) := \{\gamma:[0,1]\to M \mid \textrm{ $\gamma$ is a path with
  sitting instances}\}.
\]
\end{notation}

A useful notion of homotopy fixing end points between to paths with
sitting instances is that of a thin homotopy:

\begin{definition}[Thin homotopy] \label{def:thin-homotopy} Two paths
  $\gamma_0, \gamma_1:[0,1] \to M$ with sitting instances and the same
  endpoints in a manifold $M$ are {\sf thinly homotopic} relative
  endpoints if there is a {\sf thin homotopy} between them.  That
  is, if there is a smooth map $H:[0,1]^2 \to M$ with the following
  properties:
\begin{enumerate}
\item $H$ is a smooth homotopy from $\gamma_0$ to $\gamma_1$ relative
  the endpoints:
\[
H(s,0) = \gamma_0 (s), \quad  H(s,1) = \gamma_1 (s) \quad \textrm{for all } s\in [0,1]
\]
and
\[
H(0,t) =\gamma_0(0) = \gamma_1(0), \quad H(1,t) =\gamma_0(1) = \gamma_1(1)
\quad \textrm{for all } t\in [0,1];
\]
\item $H$ is ``thin'': 
\[
\rank(dH)_{(s,t)}\leq 1
\]
for all $(s,t)\in [0,1]^2$;
\item $H$ has sitting instances near the boundary of the square:
  $H(s,t)$ is constant in $s$ for all $(s,t)$ near
  $\{0,1\}\times[0,1]$ and is constant in $t$ for all all $(s,t)$ near
  $[0,1]\times \{0,1\}$.
\end{enumerate}
We refer to $H$ as a {\sf thin homotopy} from $\gamma_0$ to
$\gamma_1$.
\end{definition}

\begin{notation}
  We write $H:\gamma_0\Rightarrow \gamma_1$ to indicate that $H$ is a
  thin homotopy from a path $\gamma_0$ to a path $\gamma_1$.
\end{notation}

\begin{remark}\label{rmrk:2.5}
There are several points to the definition of the thin homotopy:
\begin{enumerate}
\item Two thin homotopies can be easily pasted together (vertically or
  horizontally) to give rise to a new thin homotopy.  Consequently
  ``being thin homotopic'' is an equivalence relation $\sim$ on the
  space $\calP(M)$ of paths with sitting instances.  The relation
  $\sim$ is also compatible with concatenations.

\item The pullback by a thin homotopy of any differential 2-form is zero.
  Consequently if $(P,A) \to M$ is a principal bundle with connection
  and $H:[0,1]^2 \to M $ is a thin homotopy, then the pullback bundle
  $H^*(P,A) \to [0,1]^2$ is flat.  This, in turn, implies that
  parallel transport maps defined by two thinly homotopic curves are
  equal (see Proposition~\ref{prop:equal-transport} below). In
  particular, in studying parallel transport along loops based at some
  point $x_0$, we may safely divide out by thin homotopy.
\item The collection of all loops at a point $x_0\in M$ parameterized
  by $[0,1]$ do not form a group under concatenation: for example the
  composition is not associative.  It does become associative once we
  divide out by thin homotopy.  Thus, if we want to think of parallel
  transport along loops as a representation of the ``group'' of loops,
  we need to pass to the thin fundamental group $\fund (M,x)$ (see
  Definition~\ref{def:thinfd} below).
\end{enumerate}
\end{remark}
\begin{notation}[$\calP(M)/_{\sim}$]
  Since being thinly homotopic is an equivalence relation, the
  corresponding equivalence classes make sense. We denote the
  equivalence class of a path $\gamma$ by $[\gamma]$.  We denote the
set of equivalence classes of paths in a manifold $M$ by
$\calP(M)/_{\sim}$.
\end{notation}
\begin{remark}
  In \cite{CP} thin homotopy is called ``intimacy'' to distinguish it
  from the notion of thin homotopy between piece-wise smooth paths
  introduced by Barrett ({\em op.\ cit.})  The terminology of
  Definition~\ref{def:thin-homotopy} is now standard.
\end{remark}
\begin{defprop}[Thin fundamental groupoid $\pithin(M)$]\label{defprop:thingpd}
  The concatenation of paths with sitting instances in a manifold $M$
  descends to an associative multiplication of their thin homotopy
  classes.  This multiplication gives rise to a groupoid $\pithin(M)$
  with the points of $M$ as objects and thin homotopy classes of paths
  as morphisms.
\end{defprop}
While Definition/Proposition~\ref{defprop:thingpd} is familiar to
experts, cf.\ for example \cite[Lemma~2.3]{SW}, we will recall the
argument to keep this paper self-contained. First, we fix our notation
for groupoids.

\begin{notation}\label{notation:gpds}
  Let $\Gamma$ be a groupoid, that is, a category with all morphisms
  invertible.  We denote its collection of objects by $\Gamma_0$ and
  its collection of arrows/morphisms by $\Gamma_1$.  If
  $x\xrightarrow{\gamma}y$ is an arrow in $\Gamma$, (i.e., a
  morphism from an object $x$ to an object $y$) we say that $x$ is the
  source of $\gamma$, $y$ is the target and write $x= \fs(\gamma)$,
  $y= \ft(\gamma)$.  

  The collection of pairs of composable arrows of $\Gamma$ is the
  fiber product
\[
\Gamma_1\times _{\fs,\Gamma_0, \ft } \Gamma_1\equiv \Gamma_1\times _{\Gamma_0 } \Gamma_1 
:=\{(\gamma_2, \gamma_1) \in \Gamma_1 \times \Gamma _1 \mid
\fs (\gamma_2) = \ft (\gamma_1)\}.
\]
We denote the composition/multiplication in $\Gamma$ by $\fm$:
\[
\fm:\Gamma_1\times _{\Gamma_0 } \Gamma_1  \to \Gamma_1, \qquad
\fm(\gamma_2, \gamma_1) = \gamma_2\gamma_1.
\]
In particular, we write the composition from right to left.  The
inversion map is denoted by $\fri$:
\[
\fri:\Gamma_1\to \Gamma_1,\qquad \fri(\gamma) := \gamma\inv,
\]
and the unit map is denoted by $\fu:\Gamma_0\to \Gamma_1$.

Finally, we will often write $\Gamma =\{\Gamma_1\toto \Gamma_0\}$ to
single out the collections of arrows and objects of the groupoid
$\Gamma$ together with the associated source and target maps. This
suppresses the other three structure maps from notation.
\end{notation}

\begin{proof}[Proof of Definition/Proposition~\ref{defprop:thingpd}]
  We define the set of objects of the groupoid $\pithin(M)$ to be the
  set of points of the manifold
  $M$:
\[
\pithin(M)_0:= M.
\]
We define the set of arrows of the groupoid $\pithin(M)$ to be the set
of thin homotopy classes of paths:
\[
\pithin(M)_1:= \calP(M)/_{\sim}.
\]
The source and target maps are
defined by assigning  endpoints to  classes of paths:
\[
\fs([\gamma]) := \gamma(0)\qquad \ft([\gamma]) := \gamma(1);
\]
these maps are well-defined since thin homotopies fix endpoints. The
unit map $\fu: M\to \calP(M)/_{\sim}$ assigns to each point $x\in M$
 the class of the constant path $1_x(t)\equiv x$:
\[
\fu(x):= [1_x] \quad \textrm{ for all } x\in M.
\]
Recall that if $\gamma:[0,1] \to M$ is a path, its reversal $\gamma\inv$
is defined by
$
\gamma\inv (t): = \gamma (1-t).
$
We define the inversion map $\fri:\pithin(M)_1\to \pithin(M)_1$ by
\[
\fri([\gamma]):= [\gamma\inv].
\]
The map $\fri$ is well-defined. 

The multiplication $\fm$ in the groupoid $\pithin (M)$ is defined by
concatenating representatives of the equivalence classes of paths. If
$\gamma, \tau:[0,1]\to M$ are two paths with $\gamma(0) = \tau(1),$
define $\gamma \tau$ by
\begin{equation}\label{eq:2.1}
\gamma\tau\, (t) := 
\left\{ 
\begin{array}{lrl} \tau (2t) & \mathrm{if} & t\in [0,1/2] \\
\gamma (2t-1) & \mathrm{if} & t \in [1/2,1] \\ 
\end{array} \right.  \qquad .
\end{equation}  
Note that since both $\gamma$ and $\tau$ are paths with sitting
instances, their concatenation $\gamma\tau$ is $C^\infty$.  This is one
of the reasons why working with paths with sitting instances is so
convenient. 
We then set
\[
\fm ([\gamma], [\tau]):= [\gamma\tau]
\]
for all composable classes of paths $([\gamma], [\tau]) \in (\calP
(M)/_{\sim} )\times _M (\calP (M)/_{\sim})$.  The map $\fm$ is
well-defined since thin homotopies can be concatenated.

Finally one needs to check that the five maps $\fs, \ft, \fu, \fri,
\fm$ defined above do define the structure of a groupoid on
$\pithin(M)$.  We omit this verification.
\end{proof}

\begin{remark}
  In Proposition~\ref{thm:gpoidpf} in Appendix~\ref{app:1} we check
  that $\pithin(M)$ is a diffeological groupoid.  This fact is
  well-known to experts.  We haven't been able to find a proof in literature.
\end{remark}

\begin{definition} \label{def:thinfd} Let $M$ be a manifold. We define
  the {\sf thin fundamental group } $\fund(M,x)$ of $M$ at a point $x
  \in M$ to be automorphism group of $x$ in the groupoid $\pithin(M)$:
\[
\fund(M,x) := \fs\inv (x) \cap \ft(x);
\]
it is the group of thin homotopy classes of loops at $x$ in $M$.
\end{definition}

\begin{remark}\label{cor:4.26t}
  Since for any manifold $M$ the thin fundamental groupoid
  $\pithin(M)$ is a diffeological groupoid, the automorphism groups
  $\fund(M,x)$ are diffeological groups.
\end{remark}

\section{Transport functors over manifolds}

In this section we axiomatize parallel transport in principal bundles
with connections by introducing the notion of a transport functor
(Definition~\ref{def:transp-funct}). We check that each principal
bundle with connection $(P,A)$ over a manifold $M$ gives rise to a
transport functor $\hol_M(P,A)$ (Theorem~\ref{thm:3.5*}).  Transport
functors over a manifold $M$ form a groupoid, which we denote by
$\trans(M)$. We end the section by showing
(Proposition~\ref{prop:3.17}) that the assignment of a transport
functor to a principal bundle with a connection extends to a functor
\[
\hol_M: B^\nabla G(M)\to \trans(M).
\]

\begin{definition}Let $G$ be a Lie group. A G-{\sf torsor} is a
  manifold $X$ with a free and transitive right action of a Lie
  group $X$.  In particular the map
\begin{equation}\label{eq:torsor-cond}
 X\times G\to X\times X, \qquad (x,g)\mapsto (x,x\cdot g)
\end{equation}
is a diffeomorphism.
\end{definition}

\begin{definition}[The category $\gtors$ of
  $G$-torsors] \label{def:catG-tors} Fix a Lie group $G$.  The
  collection of all $G$-torsors form a category: by definition a
  morphism from a torsor $X$ to a torsor $Y$ is a $G$-equivariant map
  $f:X\to Y$.

  We denote the category of $G$-torsors by $\gtors$.  Note that
  $\gtors$ is a groupoid.  We denote the set of morphisms in $\gtors$
  from a torsor $X$ to a torsor $Y$ by $\Hom_\gtors(X,Y)$.
\end{definition}

\begin{remark} \label{rmrk:d}
  Let $X$ be a $G$-torsor.  Since the map \eqref{eq:torsor-cond} is a
  bijection, for every pair of points $(x,y)\in X\times X$ there is a
  unique element $d= d(x,y) \in G$ so that
\begin{equation}\label{eq:torsor-cond2}
x\cdot d(x,y) = y.
\end{equation}
Since \eqref{eq:torsor-cond} is a diffeomorphism, the map $d:X\times X
\to G$ defined by \eqref{eq:torsor-cond2} is smooth.
\end{remark}
\begin{lemma}\label{lem:aut(T)} Let $G$ be a Lie group.  For any
  $G$-torsor $X$ the group of automorphisms 
\[
\Aut (X):= \Hom_\gtors (X,X)
\]
is canonically
  a Lie group. 
\end{lemma}
\begin{proof}
A choice of a point $x\in X$ gives rise to a map
\[
\psi_x:\Aut(X)\to G,\qquad  \psi_x(f) := d(x, f(x))
\]
where $d:X\times X \to G$ is the smooth map implicitly defined by 
\[
x\cdot d(x,y) = y
\] 
(q.v.\ Remark~\ref{rmrk:d} above).  It is not hard to check that
$\psi_x$ is a group isomorphism.  Hence $\psi_x$ gives $\Aut(X)$ the
structure of a Lie group.

If $y\in X$ is another choice of  base point then $y = x\cdot b$ for
some $b\in G$.  It is easy to check that
\[
\psi_{x\cdot b} = b\inv \psi_x b.
\]
Hence the Lie group structure on $\Aut(X)$ does not depend on a choice
of $x\in X$.
\end{proof}

\begin{definition}[Transport functor]\label{def:transp-funct} Let $M$
  be manifold.  A functor $F:\pithin(M)\to \gtors$ is a (parallel)
  {\sf transport functor} if, for every point $x\in M$, the map
  $F|_{\fund(M,x)}:\fund(M,x) \to \Aut(F(x))$ is a map of
  diffeological spaces (see Definition~\ref{def:mapofdiffeos}).
\end{definition}
\begin{remark}
  Recall that the group $\fund(M,x)$ is a diffeological group (q.v.\
  Remark~\ref{cor:4.26t}) and the group $\Aut(F(x))$ of automorphisms
  of the torsor $F(x)$ is a Lie group (q.v. Remark~\ref{lem:aut(T)}).
  Therefore, it makes sense to require that $F|_{\fund(M,x)}$ is
  a map of diffeological spaces, that is, smooth.
\end{remark}

We believe  that Definition~\ref{def:transp-funct} is equivalent to
Schreiber and Waldorf's definition of a transport functor \cite{SW}.
We now argue that
our definition is a conservative extension of the notion of parallel
transport defined by Barrett and by Caetano and Picken ({\em op.\ cit.}).

\begin{lemma}
  Let $M$ be a manifold and $F:\pithin(M)\to \gtors$ a functor.
  Suppose that there is a point $x\in M$ so that $F|_{\fund (M,x)}:
  \fund (M,x) \to \Aut(F(x))$ is smooth.  Then, for any point $y$ in
  the path component of $x$, the map $F|_{\fund (M,y)}: \fund (M,y)
  \to \Aut(F(y))$ is smooth.  In particular, if $M$ is connected, $F:
  \pithin(M)\to \gtors$ is a transport functor if and only there
  exists one point $x\in M$ so that $F|_{\fund (M,x)}$ is smooth.
\end{lemma}

\begin{proof}
  Choose a path $\gamma:[0,1]\to M$ with sitting instances such that
  $\gamma (0) = x$, $\gamma (1) = y$.  Then, since $\pithin(M)$ is a
  diffeological groupoid, the equivalence class $[\gamma]\in
  \calP(M)/_{\sim}$ defines a smooth map
\[
c_{[\gamma]}:\fund(M,x) \to \fund (M, y), \qquad [\tau] \mapsto
[\gamma]\inv [\tau] [\gamma]. 
\]
Similarly, we have the smooth map
\[
c_{F([\gamma])}:\Aut(F(x)) \to \Aut(F(y)), \qquad \varphi \mapsto
F([\gamma])\inv\circ \varphi \circ F([\gamma]). 
\]
Since $F$ is a functor, the diagram
\[
\xy
(-14,6)*+{\fund(M,x)}="1";
(14,6)*+{\fund(M,y)}="2";
(-14,-6)*+{\Aut(F(x))}="3";
(14,-6)*+{\Aut(F(y))}="4";
 {\ar@{->}^{c_{[\gamma]}} "1";"2"};
 {\ar@{->}^{F} "2";"4"};
 {\ar@{->}_{F} "1";"3"};
{\ar@{->}_{c_{F([\gamma])}} "3";"4"};
\endxy
\]
commutes.  That is,
\[
F|_{\fund(M,y)} =\left( c_{F([\gamma])}\right)\inv \circ
F|_{\fund(M,x)} \circ c_{[\gamma]}.
\]
Hence, $F|_{\fund(M,y)}$ is smooth.
\end{proof}

For Definition~\ref{def:transp-funct} to be reasonable, parallel
transport in principal bundles with connections have to define
transport functors; we now show that it does.

\begin{theorem}\label{thm:3.5*}
  A principal bundle with a connection $(P\xrightarrow{\pi}M, A\in
  \Omega^1(P,\fg)^G)$ over a manifold $M$ gives rise to a parallel
  transport functor
\[
  \hol_M (P,A) :\pithin(M) \longrightarrow   \gtors 
\]
  It is defined by 
\begin{equation}
  \hol_M (P,A)\, (x\xrightarrow{[\gamma]} y ):= 
 \left( \, ||^A_\gamma :P_x\to P_y\, \right)
\end{equation}
for any arrow $(x\xrightarrow{[\gamma]}y)\in \pithin(M)$.
Here $P_x$ denotes the fiber of the bundle $P\to M$ above $x\in M$ and
$||_\gamma ^A$ denotes parallel transport along a path
$\gamma$ defined by the connection $A$ (see Definition ~\ref{d:horlift} below).
\end{theorem}

To prove Theorem~\ref{thm:3.5*} we need a proposition and two lemmas.  
We start with a definition to fix our notation (cf.\
\cite[Proposition~3.1]{KN}).

\begin{definition}\label{d:horlift} Let $\pi:P\to M$ be a principal $G$-bundle with a
  connection 1-form $A$ and $\gamma:[a,b]\to M$ a path (a smooth
  curve). The {\sf horizontal lift} of $\gamma$ with respect to $A$ is
  a curve $\gamma^A_z:[a,b]\to P$ with the following three properties:
\begin{enumerate}
\item $\pi \circ \gamma^A_z = \gamma$;
\item $\gamma^A_z (a) = z$; and
\item $\frac{d}{dt} \gamma^A_z (t) \in
\ker\left(A_{\gamma^A_z(t)}\right)$ for all $t\in [a,b]$.
\end{enumerate} The {\sf parallel transport} along the path
$\gamma:[a,b]\to M$ defined by $A$ is a $G$-equivariant diffeomorphism
  (i.e., a map of $G$-torsors)
\[ 
||^A_\gamma : P_{\gamma(a)}\to P_{\gamma(b)}.
\]
It is defined by 
\[
 ||^A_\gamma (z):=
\gamma^A_z (b).
\]
for all $z\in P_{\gamma(a)}$.
\end{definition}
\begin{lemma}[Parallel transport pulls back]
\label{prop:fund_fact}\label{prop:5.17} 
Let $(P\to M, A)$, $(P'\to M', A')$ be two principal $G$-bundles with
connections and $f:P\to P'$ a $G$-equivariant map with
$A=f^*A'$. Suppose also that $f$ covers $\bar{f}:M\to M'$.  Then for
any horizontal lift $\gamma^A$ of a curve $\gamma:[a,b]\to M$ with
respect to $A$, the curve $f\circ \gamma^A$ is a horizontal lift of
the curve $\bar{f} \circ \gamma:[a,b] \to M'$ with respect to $A'$.
Consequently 
 \begin{equation}\label{eq:5.1} 
(f |_{ P_y })\circ ||^A_\gamma =
||^{A'}_{\bar{f}\circ \gamma}\circ (f|_{P_x}): P_x \to P'_{\bar{f} (y)},
\end{equation} where $x=\gamma(a)$ and $y= \gamma(b)$.
\end{lemma}
\begin{proof}
Let $\gamma^A _z:[a,b] \to P$ denote the horizontal lift of $\gamma$
starting at $z\in P_x$. Then, since $f^*A' = A$, the curve $f\circ
\gamma^A_z :[a,b]\to P'$  is a horizontal lift of $f\circ \gamma $
starting at $f(z)\in P'_{\bar{f} (x)}$.  
\end{proof}

\begin{proposition}\label{prop:equal-transport}
  Let $P\to M$ be a principal $G$-bundle with a connection 1-form $A$.
  Let $\gamma_0, \gamma_1:[0,1]\to M$ be two paths with sitting
  instances and $H:[0,1]^2\to M$ a thin homotopy from $\gamma_0$ to
  $\gamma_1$.  Let $p\in P$ be a point in the fiber above $\gamma_0(0)
  = \gamma_1 (0)$ and let $\widetilde{\gamma}_0$,
  $\widetilde{\gamma}_1$ be the corresponding horizontal lifts that
  start at $p$. Then
\[
\widetilde{\gamma}_0(1) =\widetilde{\gamma}_1 (1). 
\]
Hence the parallel transport maps $||_{\gamma_0},
||_{\gamma_1}:P_{\gamma_0(0)}\to P_{\gamma_0(1)}$ are {\em equal}.
\end{proposition}
\begin{proof} Our proof is essentially that of \cite[Section 6]{CP}).
  Pull back the principal bundle $P\to M$ to $[0,1]^2$ by the thin
  homotopy $H$. The result is a principal $G$-bundle $H^*P\to [0,1]^2$
  and a map of principal bundles
\[
\tilde{H}:H^*P\to P.
\]
Let $z'$ denote the point in the fiber of $H^*P$ above ${(0,0)}$ with
$\tilde{H}(z') = z$.  Let $\Omega$ be the curvature of $A$.  Since $H$
is a thin homotopy, $\tilde{H}^* \Omega = 0$.  Hence, the pullback
connection $A':= \tilde{H}^*A$ is flat, and $\ker A'$ defines an
integrable distribution on $H^*P$.  Since the square $[0,1]^2$ is
contractible there is a global horizontal section $\sigma:[0,1]^2 \to
H^*P$ with $\sigma(0,0) = z'$.  Then the restrictions of $\sigma$ to
the four sides of the square $[0,1]^2$ are horizontal lifts of the
corresponding curves.  By Lemma~\ref{prop:fund_fact},
$\tilde{H}$ sends horizontal lifts to horizontal lifts.  Hence,
\[
\tilde{H}( \sigma(0, t)) = \gamma^A_z(t).
\]
It also follows that 
\[
s\mapsto \tilde{H}( \sigma(s, 0))\qquad \textrm{and} \qquad s\mapsto
\tilde{H}( \sigma(s, 1))
\]
are the horizontal lifts of the constant curves $s\mapsto
\gamma(0)$ and $s\mapsto \tau(1)$; hence, they are constant curves
themselves.  Since $\tilde{H}(\sigma(0, 0)) =\tilde{H}(\sigma(1, 0))$,
the curve
\[
t\mapsto \tilde{H}( \sigma(1, t)) 
\]
is the horizontal lift of $\tau$ that starts at $z =
\gamma^A_z (a)$, that is, $\tau^A_z$.  Since $\tilde{H}(\sigma(0, 1))
=\tilde{H}(\sigma(1, 1))$, 
\[
\gamma^A_z (1)=\tau^A_z(1).
\]
\end{proof}

\begin{lemma} \label{lem:5.6} Suppose $(P\xrightarrow{\pi} M, A)$ is a
  principal $G$-bundle with connection, $F:U\times [0, 1]\to M$ a
  smooth map (so that the associated map $\check{F}:U\to \calP(M)$ is
  a plot (q.v. Definition~\ref{def:pathspacedifeol})), and $F_0,
  F_1:U\to M$ are the restrictions of $F$ to $U\times \{0\}$ and
  $U\times \{1\}$, respectively.  Then the map
\[
\Psi: F_0^*P\to F_1^* P, \qquad \Psi(u,z) :=  (u,||^A_{\check{F}(u)} (z))
\]
is smooth. Here as before $||^A$ denotes parallel transport on the principal
bundle $P \to M$ defined by  $A$.
\end{lemma}

\begin{proof}
  Recall that parallel transport $||^A_\gamma:P_{\gamma(0)} \to
  P_{\gamma(1)}$ is defined by sending $z\in P_{\gamma(0)}$ to
  $\gamma^A_z(1)$ where $\gamma_z^A: [0,1]\to P$ is the horizontal
  lift of $\gamma$ starting at $z$.   Recall also that the curve
\[
t\mapsto (t, \gamma^A_z(t))
\]
is an integral curve of a vector field $X_\gamma$ on $\gamma^*P\to
[0,1]$; $X_\gamma$ is the horizontal lift of $\frac{\partial}{\partial
  t}$ with respect to $\gamma^*A$.  Similarly let $X$ be the
horizontal lift of the vector field $(0, \frac{\partial}{\partial t})$
with respect to $F^*A$ to the bundle $F^*P\xrightarrow{F^*\pi} U\times
[0,1]$.  Then for any $(u,z) \in {F_0}^*P = \{(u,z)\in U\times P\mid
u= \pi(z)\} $ the curve
\[
t\mapsto (u, \check{F}(u)_z^A (t))
\]
is an integral curve of $X$.   Let $\Phi_1$ denote
  the time-1 flow of the vector field $X$.  Then
\[
(u,||^A_{\check{F}(u)} (z)) = \Phi_1(u,z)
\]
for all $(u,z)\in {F_0}^*P$.  It follows that 
\[
\Psi = \Phi_1|_{ {F_0}^*P},
\]
which is smooth.
\end{proof}

\begin{proof}[Proof of \protect{Theorem~\ref{thm:3.5*}}]
  We need to check that (1) $\hol_M(P,A)$ is well-defined, (2) it
  is a functor and (3) it is a transport functor in the sense of
  Definition~\ref{def:transp-funct}.

  Suppose $x\xrightarrow{[\gamma]}y$ is an arrow in $\pithin(M)$ and
  $\tau\in [\gamma]$. Then there is a thin homotopy $H:\gamma
  \Rightarrow \tau$.  By Proposition~\ref{prop:equal-transport}
  $||_\gamma = ||_\tau$.  Hence $\hol_M(P,A)$ is well-defined.

  Let $z\xleftarrow{\tau} y\xleftarrow{\gamma} x$ be a pair of
  composable paths in $M$.  Then by definition of multiplication in
  $\pithin(M)$
\[
   [\tau\gamma] = [\tau][\gamma].
\]
On the other hand, by a well-known property of parallel transport
\[
||_{\tau\gamma} = ||_\tau \circ ||_\gamma.
\]
Hence 
\[
\hol_M(P,A)([\tau\gamma]) =\hol_M(P,A)([\tau]) \hol_M(P,A)([\gamma]) .
\]
Parallel transport along constant paths is identity.  We conclude that
$\hol_M(P,A)$ is a functor.

Finally we check that for a point $x\in M$
\[
\hol_M(P,A)|_{\thinfg(M,x)}:{\thinfg(M,x)} \to \Aut(P_x)
\]
is smooth.  By Lemma~\ref{lemma:3.18} it is enough to check that for
any plot $p:U\to \Omega(M,x)$ on the space of loops at $x$ the composite map
\[
\hol_M(P,A)\circ q \circ p: U\to \Aut (P_x)
\]
is $C^\infty$.  Here $q:\Omega(M,x)\to \thinfg (M,x)$ is the quotient
map.  By construction of the $C^\infty$ structure on $\Aut(P_x)$ (see
Lemma~\ref{lem:aut(T)}) any map $L:U\to \Aut(P_x)$ is smooth if and
only if the map
\[
U\ni u\mapsto d(z, L(u)(z))\in G
\]
is smooth for some $z\in P_x$ (the map $d$ is defined in
Remark~\ref{rmrk:d}).  Thus since the map $d:P_x\times P_x\to G$ is
smooth, the smoothness of $L:U\to \Aut(P_x)$ follows from the
smoothness of the map
\[
U\to P_x,\qquad u\mapsto L(u)(z)
\] 
for some (any) choice of $z\in P_x$.  In the case we care about, this
amounts to showing that the map
\[
U\to P_x,\qquad u\mapsto ||_{p(u)}^A (z)
\]
is $C^\infty$.  This fact, in turn,  easily follows from Lemma~\ref{lem:5.6}.
\end{proof}

\begin{definition}[The category $B^\nabla G (M)$ of
    of principal bundles with connections]
    Principal $G$-bundles with connections over a manifold $M$ form a
    category $B^\nabla G$. The objects are principal bundles with
    connections, that is, pairs $(P\to M, A)$ where $P\to M$ is a
    principal $G$-bundle and $A\in \Omega^1(M,\fg)^G$ a connection.
    Morphisms are connection preserving gauge transformations.  That
    is a morphism from $(P\xrightarrow{\pi}M,A)$ to
    $(P'\xrightarrow{\pi'}M,A')$ is a $G$-equivariant map $f:P\to P'$
    with $\pi'\circ f =\pi$ and $f^*A' = A$.    
\end{definition}

\begin{remark}
The category $B^\nabla G (M)$ is a groupoid since every gauge
    transformation is automatically invertible.
\end{remark}

\begin{definition}[The category $\trans(M)$ of
    $G$-transport functors]
    Fix a Lie group $G$ and a manifold $M$. Transport functors on $M$
    form a category $\trans(M)$: the objects are transport functors
    (q.v.\ Definition \ref{def:transp-funct}) and morphism are
    arbitrary (!) natural transformations.
\end{definition}

\begin{remark}
  Since the category of $G$-torsors is a groupoid, a natural
  transformation between two transport functors is automatically a
  natural isomorphism.  Hence, $\trans(M)$ is automatically a groupoid.
\end{remark}

\begin{proposition} \label{prop:3.17}
For a manifold $M$ the assignment 
\[
(P,A)\mapsto \hol_M(P,A)
\]
of a transport functor to a principal bundle with a connection extends
to a functor
\[
\hol_M: B^\nabla G(M)\to \trans(M)
\]
from the category of principal $G$-bundles with connections over $M$
to the category of transport functors.
\end{proposition}

\begin{proof}
  Let $f:(P,A)\to (P',A')$ be a morphism in $B^\nabla G(M)$.  We want
  to define natural transformation $\hol_M(f):\hol_M(P,A)\Rightarrow
  \hol_M(P',A')$.  Let $x\xrightarrow{[\gamma]}y$ be an arrow in
  $\pithin(M)$.  By \eqref{eq:5.1} the diagram
\[
\xy
(-10,6)*+{P_x}="1";
(10,6)*+{P_y}="2";
(-10,-6)*+{P'_x}="3";
(10,-6)*+{P'_y}="4";
 {\ar@{->}^{||^A_\gamma} "1";"2"};
 {\ar@{->}_{f|_{P_x}} "1";"3"};
 {\ar@{->}_{||^{A'}_\gamma} "3";"4"};
{\ar@{->}^{f|_{P_y}} "2";"4"};
\endxy
\]
commutes.  Since $||^A_\gamma = \hol_M(P,A)\,([\gamma])$ and
$||^{A'}_\gamma = \hol_M(P',A')\,([\gamma])$ it follows that the
assignment
\[
x\mapsto f|_{P_x}
\]
is a natural transformation from $\hol_M(P,A)$ to $\hol_M(P',A')$.  We denote it by $\hol_M(f)$.    It is easy to check that the map
\[
\hol_M:B^\nabla G(M)\to \trans (M), \qquad 
 \left( (P,A)\xrightarrow{f} (P',A')\right) \mapsto \left( \hol_M(P,A)\xRightarrow{\hol_M(f)}  \hol_M(P',A')\right)
\]
sends identity maps to identity natural transformations and preserves
composition. In other words $\hol_M$ is a functor.
\end{proof}

\section{Equivalence of the categories of principal bundles with
  connections and of transport functors}
\label{proofsection}

The goal of this section is to prove that the functor $\hol_M:
B^\nabla G(M)\to \trans(M)$ constructed in Proposition~\ref{prop:3.17}
is an equivalence of categories. Note that since our definition of
transport functors is different from the one proposed by Schreiber and
Waldorf in \cite{SW} this is not an alternative prove of an analogous
result in \cite{SW}.  Our proof is in the same
in spirit as Barrett's original proof \cite{B}.  The details are
necessarily different since we are carefully keeping track of
morphisms.   We start by formally stating the theorem in question:

\begin{theorem}\label{thm:sub_main} For every manifold $M$ the
  holonomy functor 
\[
\hol_M: B^\nabla G(M) \to \trans (M)
\]
constructed in Proposition~\ref{prop:3.17} is an equivalence of categories.
\end{theorem}

We first reduce the proof of Theorem~\ref{thm:sub_main} to the case
where the manifold $M$ is connected.  This is done to simplify the
proof of Lemma~\ref{lemmaB} below.  One can also prove Lemma~\ref{lemmaB}
directly for arbitrary manifolds at a cost of additional fiddling.

\begin{lemma}\label{lem:6.2}
  Suppose the functor $\hol_M: B^\nabla G(M)\to \trans(M)$ is an
  equivalence of categories for all {\em connected} manifolds $M$.
  Then $\hol_M$ is an equivalence of categories for {\em any} manifold
  $M$.
\end{lemma}
\begin{proof} Fix a manifold $M$.
We may assume that the set of connected components of $M$ is indexed by a set $I$. Then 
\[
M = \bigsqcup _{\alpha \in I} M_\alpha,
\]
where the $M_\alpha$'s are connected components of $M$.  Observe that 
\[
 B^\nabla G(M) = \prod_{\alpha \in I} B^\nabla G(M_\alpha) \quad \text{ and}
\]
\[
 \trans(M) = \prod_{\alpha \in I} \trans(M_\alpha),
\]
where $\prod$ denotes product of
categories.  It is not hard to check that
 the diagram 
\begin{equation}\label{eq:6blah}
\xy
(-22,6)*+{\prod_{\alpha \in I} B^\nabla G(M_\alpha)}="1";
(22,6)*+{\prod_{\alpha \in I} \trans(M_\alpha)}="2";
(-22,-6)*+{ B^\nabla G(M_\beta)}="3";
(22,-6)*+{\trans(M_\beta)}="4";
 {\ar@{->}^{\hol_M} "1";"2"};
 {\ar@{->}^{\hol_{M_\beta}} "3";"4"};
 {\ar@{->}_{} "2";"4"};
{\ar@{->}_{} "1";"3"};
\endxy
\end{equation}
commutes for all $\beta\in I$.  The result follows from these
observations.
\end{proof}
To prove Theorem~\ref{thm:sub_main} for connected manifolds we
introduce a third category, $B^pG(M)$. To properly define this
category we recall the notion of a (left) groupoid action (see \cite{MacK}, for example).
\begin{definition}
  A (left) {\sf action of a groupoid $\Gamma=\{\Gamma_1\toto \Gamma_o\}$ on
    a set} $M$ consists of a map $\alpha:M\to \Gamma_0$ called the {\sf 
    anchor} and a map
\[
a:\Gamma_1\times_{\fs, \Gamma_0, \alpha }M\to M,\qquad a(\gamma, m) =
\gamma\cdot m
\]
called the {\sf action}  so that
\begin{enumerate}
\item $a(\gamma \cdot m) = \ft(\gamma)$ for all $\gamma \in \Gamma_1$
  and $m\in M$ with $\fs (\gamma) = a(m)$;
\item the identity arrows act trivially: $1_{\alpha(m)}\cdot m = m$
  for all $m\in M$; and
\item $\tau \cdot (\gamma \cdot m) = (\tau \gamma)\cdot m$ for any two
  composable arrows $\tau, \gamma$ and any point $m\in M$ with $\fs(\gamma) =
  \alpha(m)$.
\end{enumerate}
Here, as before, $\ft$ and $ \fs$ are the target and the source maps of the
groupoid $\Gamma$, respectively.
\end{definition}

\begin{definition}
  Fix a Lie group $G$ and a manifold $M$.  The {\sf category of
    $B^pG(M)$ of principal $G$-bundles over $M$ with smooth actions of
    the thin fundamental groupoid $\pithin(M)$} is defined as follows:
  the objects are smooth (left) actions $a:\calP(M)/_{\sim}\times P\to
  P$ of $\pithin(M)$ on principal $G$-bundles $P\to M$ that commute
  with the right actions of $G$.  Morphisms are $\pithin(M)\times
  G$-equivariant maps. That is, they are maps $f:P\to P'$ of principal
  $G$-bundles so that
\[
f([\gamma]\cdot z) =  [\gamma]\cdot f(z)
\]
 for all $([\gamma], z)\in \calP(M)/_{\sim}\times
  P$.
\end{definition}

There are several reasons for introducing this category.  Note first:
\begin{lemma} \label{lemma:3.14} A connection 1-form $A$ on a
  principal $G$-bundle $P\xrightarrow{\pi}M$ defines a smooth action
  $a$ of the thin fundamental groupoid
  $\pithin(M):=\{\calP(M)/_{\sim}\toto M\}$ on the manifold $P$
  relative to the anchor map $\pi$.
\end{lemma}
\begin{proof}
We define the map   $a: \calP(M)/_{\sim}\times_M P \to P$ by
\[
a([\gamma], z): = ||_\gamma (z)
\]
for all pairs $([\gamma], z)\in \calP(M)/_{\sim} \times_{\fs, M, \pi}
P$, that is, for all pairs $([\gamma], z)$ with $\fs([\gamma]) =
\gamma(0) = \pi(z)$.  Since $||_{1_{\pi(z)}}(z) = z$ and, of any pair
of composable paths $\gamma$ and $\tau$, $||_\gamma \circ ||_\tau =
||_{\gamma\tau}$, $a$ is indeed an action.

To check that the action $a$ is smooth, pick a pair of plots
$\underline{p}:U\to \calP(M)/_{\sim}$ and $r:U\to P$ with $\fs \circ
\underline{p} = \pi\circ r$ (q.v. Construction~\ref{constr:fpdiffeo},
the construction of the fiber product diffeology).  We need to show
that $a\circ (\underline{p},r):U\to P$ is $C^\infty$.  We may assume
that $\underline{p}$ has a global lift $p:U\to \calP(M)$.  Then
\[
a\circ (\underline{p}, r) (u) = ||_{p(u)} (r(u)).
\]
Let $F:U\times [0,1] \to M$ denote the map associated to the plot $p$:
$F(u,t)  := p(u)(t)$. By Lemma~\ref{lem:5.6}, the map $F_0^*P \to
F_1^*P$ given by $(u,z) \mapsto ||_{p(u)} z$ is smooth.  Since $r:U\to
P$ is smooth, $u\mapsto ||_{p(u)} (r(u))$ is smooth as well.
\end{proof} 

In fact, much more is true.  We will show:
\begin{itemize}
 \item The map that sends a principal
  bundle with connection $(P,A)$ to an action of $\pithin(M)$ on $P$
  extends to an {\em isomorphism} of categories $\tp: B^\nabla G(M)\to
  B^pG (M)$; see Lemma~\ref{lemmaA}.

\item  There is a natural equivalence of
  categories $\rep:B^pG(M) \to \trans (M)$ so that the diagram
\begin{equation} \label{eq:reptp=hol}
\xy
(-15,6)*+{B^\nabla G(M)}="1";
(15,6)*+{\trans(M)}="2";
(0,-10)*+{B^pG(M)}="3";
 {\ar@{->}^{\hol} "1";"2"};
 {\ar@{->}_{\tp} "1";"3"}; 
{\ar@{->}_{\rep} "3";"2"};
\endxy
\end{equation}
commutes; see Lemma~\ref{lemmaB} and Remark~\ref{rmrk:6.*8}.
\end{itemize}
Clearly these two facts imply that $\hol_M: B^\nabla G(M)\to
\trans(M)$ is an equivalence of categories and thereby
Theorem~\ref{thm:sub_main}.  We now proceed to define the relevant
functors.

\begin{definition}[The functor $\rep:B^pG(M) \to \trans(M)$]
\label{def:rep}
 Given an
  action $a: \calP(M)/_{\sim}\times_M P\to P$ of the thin fundamental
  groupoid $\pithin(M)$ on a principal $G$-bundle $P\to M$, define
  the associated transport functor $\rep(a):\pithin(M)\to \gtors$ by
\[
\rep(a)\, (x\xrightarrow{[\gamma]} y):= P_x \xrightarrow{ [\gamma]\cdot-} P_y.
\]
Here $[\gamma]\cdot -$ denotes the action of $[\gamma]$ on the points
of the fiber $P_x$.  Given two actions $a$ and $a'$ on principal $G$-bundles
$P$ and $P'$ over $M$ and a $(\pithin(M)\times G)$-equivariant map $f:P\to P'$,
define a natural isomorphism $\rep(f) :\rep(a)\Rightarrow \rep(a')$
by
\[
\rep(f)_x:= f|_{P_x}: P_x \to P'_{x}
\]
for all $x\in M$.
\end{definition}

\begin{remark}
  Note that, since $a$ is a action, $\rep(a)$ is indeed a
  functor. Furthermore, since $a$ is a {\em smooth} action, the map
\[
\fund(M,x)\times P_x \to P_x,\qquad ([\gamma], z)\mapsto [\gamma]\cdot z
\]
is smooth.  Hence, since the category of diffeological spaces is
Cartesian closed \cite{LaubingerDiss, BH}, the adjoint map
$\fund(M,x) \to \Aut (P_x)$ is smooth.  Therefore $\rep(a)$ is a 
  transport functor.  We conclude that the functor $\rep$ is
well-defined.
\end{remark}

\begin{definition}[the functor $\tp: B^\nabla G(M) \to B^p G (M)$]
  \label{def:tp}
  Given a principal $G$-bundle with connection $(P,A)$ define an
  action $\tp(P,A):\calP(M)/_{\sim}\times_M P\to P$ of $\pithin(M)$ on $P$ by
\[
\tp(P,A)( [\gamma], z) := ||_\gamma^A(z)
\]
for all $([\gamma], z)\in \calP(M)/_{\sim}\times_M P$ (q.v.\
Lemma~\ref{lemma:3.14}).

Equation \eqref{eq:5.1} implies that a map $f:(P,A)\to (P',A')$ of
principal $G$-bundles with connection intertwines the actions $\tp(P,A)$
and $\tp(P',A')$.  Hence $f:P\to P'$ is also a morphism in $B^pG(M)$.
\end{definition}

\begin{remark}\label{rmrk:6.*8}
  It follows easily from the definitions that $\rep\circ \tp = \hol$,
  that is, the diagram \eqref{eq:reptp=hol} commutes.
\end{remark}
We now state:

\begin{lemma}\label{lemmaA} For a  manifold $M$ 
  the functor $\tp:B^\nabla G(M) \to B^pG(M)$ of
  Definition~\ref{def:tp} is an isomorphism of categories.
\end{lemma}
\begin{lemma}\label{lemmaB} For a connected manifold $M$ 
  the functor $\rep:B^pG(M)\to \trans (M)$ of
  Definition~\ref{def:rep} is an equivalence of categories.
\end{lemma}

The proofs of these two lemmas take up the rest of the section. Our proof
of Lemma~\ref{lemmaA} is surprisingly fiddly.  We first describe a procedure for building a smooth family of paths
in $\calP(M)/_{\sim}$ from any path on $M$.

\begin{construction}\label{c:upsilon}
  Let $(a,b)$ be an interval containing $0$ and $\gamma:(a, b) \to M$
  a smooth path in a manifold $M$ for $(a,b)$.  We construct a smooth
  path (a plot) 
\[
\Upsilon: (a,b) \to \calP(M)
\]
 as follows.  Let $\beta
  \in C^\infty ([0,1])$ be the function of
  Remark~\ref{rmrk:5.3333}. Define
\begin{equation*}
 \Upsilon (s) \, (t) := \gamma (s\beta(t))
\end{equation*}
for all $t\in [0,1]$, $s\in (a,b)$.  Since the map 
$\hat{\Upsilon}: (a,
b)\times [0,1]\to M$ given by 
\[
 \hat{\Upsilon}(s,t) := \gamma
(s\beta(t)) 
\]
is smooth, the map $\Upsilon$ is indeed a smooth path in the
diffeological space $\calP(M)$ of paths in $M$.  Note that the path
$\Upsilon$ satisfies $\Upsilon(0) = 1_x$, where $1_x$ denotes the
constant path at $x= \gamma (0)$.  Additionally $\Upsilon(s)(0) = x$
and $\Upsilon(s)(1) = \gamma (s)$ for all $s$.

Note also that the composite $q\circ\Upsilon $ is a smooth path in the
space $\calP(M)/_{\sim}$ of arrows of the groupoid $\pithin(M)$ (here,
as before, $q:\calP(M)\to \calP(M)/_{\sim}$ is the quotient map).
\end{construction}

\begin{lemma} \label{lem:6.12}
  For any manifold $M$ the functor $\tp: B^\nabla G(M) \to B^p G(M)$
  is bijective on morphisms and injective on objects.
\end{lemma} 
\begin{proof}
  We first argue that the functor $\tp$ is bijective on morphisms.
  That is, for any two principal $G$-bundles with connection $(P,A)$ and
  $(P',A')$ over the manifold $M$, the map
\begin{equation} \label{eq:6.3}
\tp: \Hom_{B^\nabla G(M)}((P,A),(P',A')) \to \Hom_{B^pG(M)} (\tp(P,A),\tp (P',A'))
\end{equation}
is a bijection.  A morphism in $\Hom_{B^pG(M)} (\tp(P,A),\tp (P',A'))$
is a map of principal  bundles $f:P\to P'$ which is also $\pithin(M)$
equivariant (with the two actions defined by the connections $A$ and
$A'$, respectively).  A morphism in $\Hom_{B^\nabla
  G(M)}((P,A),(P',A'))$ is a map of principal bundles $h:P\to P'$ covering the identity on $M$ with
$h^*A' = A$.  Thus the map \eqref{eq:6.3} is a bijection if and only if any
$\pithin(M)$-equivariant map of principal bundles preserves the
respective connections.

Given path $\gamma:[a,b]\to M$, denote by $\gamma^A_z$ its horizontal
lift to $P$ with respect to the connection $A$ which starts at $z\in
P_{\gamma(a)}$.  To prove that a $\pithin(M)$-equivariant map of
principal bundles $f:P\to P'$ preserves connections, it is enough to
show that
\begin{equation}\label{eq:5.5}
f\circ \gamma_z^A = \gamma^{A'}_{f(z)}
\end{equation}
for all curves $\gamma $ in $M$.
This is because given a point $x\in M$ and a vector $v\in T_xM$, we can
choose $\gamma:(-\epsilon, \epsilon) \to M$ with $\gamma (0) = x$ and
$\dot{\gamma}(0) = v$.  Then $\dot{\gamma}^A_z(0)$ is the horizontal
lift of $v$ to $T_zP$ with respect to $A$.  So if  \eqref{eq:5.5} holds,
then 
\[
Df_z (\dot{\gamma}_z^A (0) ) = \dot{\gamma} _{f(z)}^{A'} (0).
\]
Consequently, since $f$ is $G$-equivariant, we have to have 
\[
f^*A' = A.
\]
To prove that \eqref{eq:5.5} holds,  first observe that if $\tau: (-\epsilon,
\epsilon) \to P$  is a curve with $\pi\circ \tau = \gamma$, $\tau (0)
= z$, and 
\[
\tau(s)  = ||_{\gamma|_{ [0,s]}}(z)
\]
for all $s\in (-\epsilon, \epsilon) $, then $\tau$ {\em is} the
horizontal lift $\gamma _z^A$ of $\gamma$.  This is because $\gamma
_z^A$ is the unique curve with $\gamma_z^A (0) = z$ and $\gamma_z^A
(s) =  ||_{\gamma|_{ [0,s]}}(z)$.

By Construction ~\ref{c:upsilon} a path
$\gamma:(-\epsilon,\epsilon)\to M$ gives rise to a a smooth path
$\Upsilon:(-\epsilon,\epsilon)\to \calP(M)$ with $\Upsilon(0)=1_x$
(the constant path at $x=\gamma(0)$), $\Upsilon(s)(0)=x$, and
$\Upsilon(s)(1)=\gamma(s)$ for every $s\in(-\epsilon,\epsilon).$
Consequently
\[
\tau(s) := [\Upsilon(s)]\cdot z
\]
is a smooth path in the principal bundle $\pi:P\to M$ with $\tau(0) =
[1_x]\cdot z = z$ and with $\pi (\tau(s)) = \gamma(s)$.  Moreover, by
definition of the action of $\pithin(M)$ on $P$, the point $\tau(s)\in
P$ is the image of $z$ under parallel transport along the curve
$\nu(t):= \gamma (s\beta(t))$. 
 Since $\nu(t)$ is a ``reparameterization''
of $ \gamma|_{[0,s]}$, $\tau (s) = ||_{\gamma_{[0,s]}}$ for all $s$.
We conclude that the curve $\tau$ is the horizontal lift $\gamma_z^A$
of $\gamma$.  Since $f$ is $\pithin(M)$-equivariant,
\[
f(\tau(s)) = f([\Upsilon(s)]\cdot z) = [\Upsilon(s)] \cdot f(z).
\]
By the same argument, the curve $s\mapsto [\Upsilon(s)] \cdot f(z)$ is
the horizontal lift $\gamma_{f(z)}^{A'}$ of $\gamma$ to $P'$.  Hence 
\eqref{eq:5.5} holds and we conclude that the map \eqref{eq:6.3} is 
a bijection.

The argument above also implies that if $A$ and $A'$ are two connections on
a principal $G$-bundle $P\to M$ that define {\em the same} action of
$\pithin(M)$, then $(id_P)^*A' = A$, that is, $A= A'$.  Therefore, the
functor $\tp$ is injective on objects.
\end{proof} 

It remains to show that the functor $\tp$ is surjective on objects;
that is, an action of the thin fundamental groupoid $\pithin(M)$
on a principal $G$-bundle $P\to M$ defines a connection.  We
prove this in a series of lemmas and corollaries.  The first is a
variant of a lemma due to Barrett \cite{B}.

\begin{lemma}\label{lem:Barrett}
  Let $Q\to \R^n$ be a principal $G$-bundle with a smooth action of the thin
  fundamental groupoid $\pithin(\R^n)$ and $p:(-\epsilon, \epsilon)
  \to \fund(\R^n,x)\subset \calP(\R^n)/_{\sim}$ a smooth family of
  (thin homotopy classes of) loops with $p(0) = [1_x]$, the class of
  the constant loop at $x\in \R^n$.  Then 
\[
\left.\frac{d}{ds}\right|_0 (p(s)\cdot z) = 0
\]
for any point $z$ in the fiber of $Q$ above $x$.
\end{lemma}

\begin{proof}
  We may assume that the path $p:(-\epsilon, \epsilon)
  \to \calP(\R^n)/_{\sim}$ has a global lift
  $\tilde{p}:(-\epsilon, \epsilon) \to \calP(\R^n)$.   Then
\[
(\tilde{p}(s))(t) = (\hat{p}_1(s,t),\ldots, \hat{p}_n(s,t))
\]
for some smooth functions $\hat{p}_j: (-\epsilon, \epsilon) \times
[0,1]\to \R$ with $\hat{p}_j (0,t) = x_j$ for all $t$.
Consider the map  $P:(\epsilon, \epsilon)^n\to \calP(\R^n)$  defined by
\[
P(s_1,\ldots, s_n)(t) = (\hat{p}_1(s_1,t),\ldots, \hat{p}_n(s_n,t));
\]
it is also a plot for $\calP(\R^n)$.  We then have
\[
p(s) = q\circ P\,(s,\ldots,s)\, =[P(s,\ldots,s)], 
\]
where $q:\calP(\R^n) \to \calP(\R^n)/_{\sim}$ is the quotient map.
Since the action of $\pithin(\R^n)$ on the bundle $Q$ is smooth, the map
\[
F:(-\epsilon, \epsilon)^n\to Q, \qquad F(s_1,\ldots, s_n) :=
[P(s_1,\ldots,s_n)]\cdot z
\]
is smooth.  For each $s\in (-\epsilon, \epsilon)$, the paths $s\mapsto
P(s,0,\ldots, 0), s\mapsto P(0,s, \ldots, 0), \ldots ,s\mapsto P(0,
\ldots,0, s) $ are all thinly homotopic to the constant path $1_x$.

It follows that
\[
\frac{\partial F}{\partial s_j} (0,\ldots, 0) = 0
\]
for all $j$.   

By the chain rule
\[
\left.\frac{d}{ds} (p(s)\cdot z)\right|_0= \left.\frac{d}{ds}\right|_0
F(s,\ldots, s) =\sum_{j=0}^n \frac{\partial F}{\partial s_j} (0,\ldots, 0)\cdot 1
= 0.
\]
\end{proof}
\begin{corollary}\label{cor:6.14}
  Let $Q\to \R^n$ be a principal $G$-bundle with a smooth action of the thin
  fundamental groupoid $\pithin(\R^n)$ as above and $\tau, \tau':
  (-\epsilon, \epsilon) \to \fund(\R^n,x)\subset \calP(\R^n)/_{\sim}$
  be two smooth maps with $\tau(0) = \tau'(0) = [1_x]$, $x\in \R^n$, and 
$\ft(\tau (s)) = \ft(\tau'(s))$ for all $s\in  (-\epsilon, \epsilon) $.
Then
\[
\left.\frac{d}{ds}\right|_0 (\tau (s)\cdot z) = 
\left.\frac{d}{ds}\right|_0 (\tau'(s)\cdot z)
\]
for any point $z$ in the fiber of $Q$ above $x$.
\end{corollary}

\begin{proof}
  By assumption, $p(s) = \tau(s)\inv \tau'(s)$ is a loop for each
  $s\in (-\epsilon, \epsilon)$ and $p(0)$ is the constant loop
  $[1_x]$.  Note that 
\[
\tau' (s) = \tau (s)\cdot \tau (s) \inv\cdot \tau' (s) = 
\tau (s) \cdot p(s).
\]
Hence 
\[
\left.\frac{d}{ds}\right|_0 (\tau'(s)\cdot z) = \left.\frac{d}{ds}\right|_0 
( \tau (s) \cdot (p(s) \cdot z))
\]
Since the action of $\pithin(M)$ on $Q$ is smooth the map 
 $F:(-\epsilon, \epsilon)^2 \to Q$ defined by 
\[
F(s_1, s_2) = \tau (s_1) \cdot (p(s_2) \cdot z)
\]
is smooth.  By the chain rule
\[
\left. \frac{dF}{ds} (s,s) \right|_0 = \frac{\partial F}{\partial s_1}
(0,0) + \frac{\partial F}{\partial s_2} (0,0).
\]
Hence
\[
\left.\frac{d}{ds}\right|_0 (\tau'(s)\cdot z) = \left.\frac{d}{ds}\right|_0 
( \tau (s) \cdot (p(0) \cdot z)) +  \left.\frac{d}{ds}\right|_0 
( \tau (0) \cdot (p(s) \cdot z)).
\]
By Lemma~\ref{lem:Barrett} the second term is zero, and we are done.
\end{proof}

\begin{lemma} \label{lemma:5.34} Let $\pi:Q\to \R^n$ be a principal
  $G$-bundle. A smooth action $a:(\calP(\R^n)/_{\sim})\times_{\fs, \R^n, \pi}
  Q\to Q$ of the thin fundamental groupoid $\pithin(\R^n)$ on $Q$
  defines a connection $A$ (which is necessarily unique by
  Lemma~\ref{lem:6.12}) so that for any path $\gamma:[0,1] \to \R^n$ and
  any point $z$ in the fiber $Q_x$ above $x= \gamma (0)$
\[
[\gamma]\cdot z = \gamma_z^A(1).
\]
Here, as before, $\gamma_z^A:[0,1]\to Q$ is a lift of $\gamma$ to $Q$
which is $A$-horizontal and starts at the point $z$.
\end{lemma}

\begin{proof}
  As we have seen in the proof of Lemma~\ref{lem:6.12}, an
  action of $\pithin(\R^n)$ on a principal bundle $Q\to \R^n$ allows
  us to lift curves of the form $\gamma :(-\epsilon, \epsilon)\to
  \R^n$ to $Q$.  Thus, given a point $x\in \R^n$ and a point $z\in Q$
  in the fiber above $x$, it is tempting to define the horizontal
  subspace $\calH Q_z \subset T_z Q$ by
\begin{equation} \label{eq:5.7}
  \calH Q_z := \left\{ 
    \left. \left. \frac{d}{ds}\right|_{s=0} [\Upsilon (s)]\cdot z \quad 
    \right| \quad  \Upsilon(s)(t)  = \gamma (s\beta(t)),\quad 
\gamma:(-\epsilon, \epsilon) \to \R^n, \gamma (0) = x
  \right\}.
\end{equation}
Here as before $\Upsilon$ is the smooth path in $\calP(M)$ constructed
from the path $\gamma$ (q.v.\ Construction ~\ref{c:upsilon}).
However, it is not clear that \eqref{eq:5.7} defines a vector space.
Nor is it clear that curves of the form $s\mapsto [\Upsilon(s)]\cdot
z$ are tangent to the purported distribution $\calH$ when $s\not =0$.
We therefore proceed a little differently.

Consider the map 
\begin{equation}\label{eq:6.777}
\sigma: \R^n \times \R^n \to \calP(\R^n),  
\qquad \sigma (x,y) (t) := x + \beta(t) (y-x),
\end{equation}
for all $t\in [0,1]$.  Here as before $\beta\in C^\infty ([0,1])$ is
the function of Remark~\ref{rmrk:5.3333}. The map $\sigma$ is
smooth since the map $\hat{\sigma} (x,y, t) := x + \beta(t) (y-x)$ is
smooth.  We denote the induced plot on the space of arrows of the thin
fundamental groupoid $\pithin(\R^n)$ by $[\sigma]$.  By
construction, $[\sigma(x,y)]$ is the thin homotopy class of the
straight line paths from $x$ to $y$ with $[\sigma(x,x)] =
[1_x]$, the class of the constant path.

Next consider the smooth map 
\[
F:\R^n\times Q\to Q, \qquad F(y,z):= [\sigma(\pi(z), y)]\cdot z.
\]
By construction 
\[
\pi (F(y,z)) = y.
\]
Thus for each fixed $z\in P$ the map 
\[
F(\cdot, z):\R^n \to Q
\]
is a smooth section of $\pi:P\to \R^n$
with 
\[
F(\pi(z), z) =
[\sigma (\pi(z), \pi(z))]\cdot z = [1_{\pi(z)}]\cdot z = z.
\]
  Denote the derivative of this section at $y\in \R^n$ by $\partial _1
  F(y,z)$.  By definition, the derivative $\partial _1 F(\pi(z),z)$
  is a linear map from $T_{\pi(z)}\R^n$ to $T_z Q$.  We define
\[
\calH Q_z:= \mathrm{image }\, \left( \partial _1
    F(\pi(z),z): T_{\pi(z)}\R^n \rightarrow T_zQ\right).
\]
Since $y \mapsto F(y,z)$ is a section of $\pi$, the image of its
differential at every point is a subspace complementary to the
vertical bundle of $\pi: Q\to \R^n$.  Since $F:\R^n\times Q\to Q$ is
smooth, the map $z\mapsto \partial _1 F(\pi(z), z)$ is smooth.  It
follows that the subspace $\calH Q_z$ depends smoothly on $z\in Q$.
Since the action of $\pithin(\R^n)$ on $Q$ commutes with the action of
$G$,
\[
F(y,z\cdot g)=  [\sigma(\pi (z\cdot g), y)] \cdot (z\cdot g) 
= ([\sigma(\pi (z), y)] \cdot z)\cdot g = F(y,z) \cdot g 
\]
for any $y\in \R^n$, $z\in Q$ and $g\in G$. Consequently, $\calH Q\subset
TP$ is a $G$-invariant distribution.  
Denote the corresponding connection 1-form by $A$.  It remains to
check that $ [\gamma]\cdot z = \gamma_z^A(1)$ for a path
$\gamma:[0,1] \to \R^n$ and any point $z$ in the fiber $Q_{\gamma(0)}$ above $
\gamma (0)$.

Since $\gamma$ is smooth, there is $\epsilon>0$ and an extension of
$\gamma$ to a smooth map from $(-\epsilon, 1+ \epsilon )$ to $\R^n$.
We denote the extension by the same symbol $\gamma$.  
Consider the corresponding path $\Upsilon: (-\epsilon, 1+\epsilon) \to
\calP(\R^n)$ satisfying $\Upsilon(s)(0)=\gamma(0)$ and
$\Upsilon(s)(1)=\gamma(s)$ for any $s\in(-\epsilon, 1+\epsilon)$
(Construction ~\ref{c:upsilon}).

We want to show that
\[
\left. \frac{d}{ds} \right|_{s= s_0} [\Upsilon(s)]\cdot z \in
\calH_{[\Upsilon(s)]\cdot z} Q
\]
for any $s_0 \in [0,1]$, and any $z\in Q_{\gamma(0)}$.   Note that, by
Construction ~\ref{c:upsilon},
\[
\pi ([\Upsilon(s)]\cdot z) = \ft ([\Upsilon (s)])  = \Upsilon (s)(1) =
\gamma (s).
\]
Set $z_0:= [\Upsilon (s_0)] \cdot z$.  Then:
\begin{align*}
\left. \frac{d}{ds} \right|_{s= s_0} [\Upsilon(s)]\cdot z 
&=\left. \frac{d}{ds} \right|_{s= 0} [\Upsilon(s+ s_0)]\cdot z \\ 
&=\left. \frac{d}{ds} \right|_{s= 0} ([\Upsilon(s+ s_0)][\Upsilon (s_0)]\inv
)\cdot( [\Upsilon (s_0)]
\cdot z) \\ 
&=\left. \frac{d}{ds} \right|_{s= 0} ([\Upsilon(s+ s_0)][\Upsilon (s_0)]\inv
  )\cdot z_0
\end{align*}
We next consider two paths in $\calP(\R^n)/_{\sim}$ defined on
$(-\epsilon, \epsilon)$:
\[
\tau (s) = [\Upsilon(s+ s_0)][\Upsilon (s_0)]\inv
\]
and 
\[
\tau'(s) =[\sigma (\gamma (s_0), \gamma (s+s_0))],
\]
where $\sigma$ is defined by \eqref{eq:6.777}.  These two paths
 satisfy the hypotheses of Corollary~\ref{cor:6.14}.
Hence
\begin{equation}\label{eq:6.8}
\left. \frac{d}{ds} \right|_{s= 0} \tau (s) \cdot z_0 =
\left. \frac{d}{ds} \right|_{s= 0} \tau' (s) \cdot z_0. 
\end{equation}
By construction of the path $\tau'$ and the distribution $\calH$, the
right hand side of \eqref{eq:6.8} is a vector in $\calH _{z_0}
Q=\calH_{[\Upsilon(s)]\cdot z} Q$.  Hence $\left. \frac{d}{ds}
\right|_{s= s_0} [\Upsilon(s)]\cdot z =\left. \frac{d}{ds} \right|_{s=
  0} \tau (s) \cdot z_0\in \calH_{[\Upsilon(s)]\cdot z} Q$ as well.
\end{proof}

\begin{lemma} \label{lemma:5.34+} Let $\pi:Q\to M$ be a principal
  $G$-bundle. An action $a:(\calP(M)/_{\sim})\times_{\fs, \R^n, \pi}
  Q\to Q$ of the thin fundamental groupoid $\pithin(M)$ on $Q$ defines
  a unique connection $A$ so that for any path $\gamma:[0,1] \to M$
  and any point $z$ in the fiber $Q_x$ above $x= \gamma (0)$
\[
[\gamma]\cdot z = \gamma_z^A(1).
\]
\end{lemma}

\begin{proof}
  For any point $x\in M$ there is a coordinate chart $\varphi:U\to
  \R^n$ on $ M$ with $x\in U$ and $\varphi$ a diffeomorphism.  Since
  the desired connection $A$ would have to be unique (see
  Lemma~\ref{lem:6.12}), it is enough to define it on restrictions
  $Q|_U$ where $U\subset M$ is a domain of a diffeomorphism
  $\varphi:U\to \R^n$.
 
  The action of $\pithin(M)$ on $Q$ defines a smooth action of
  $\pithin(U)$ on $Q|_U$.  The diffeomorphism $\varphi$ allows us to
  transfer this action to an action of $\pithin(\R^n)$ on the
  principal $G$-bundle $(\varphi\inv )^*Q\to \R^n$.  By
  Lemma~\ref{lemma:5.34}, the action of $\pithin(\R^n)$ defines a
  connection on $(\varphi\inv )^*Q\to \R^n$.  Its pullback to
  $Q|_U\to U$ defines the restriction of the desired connection $A$ to
  $Q|_U$.
\end{proof}

\begin{proof}[Proof of \protect{Lemma~\ref{lemmaA}}]
  By Lemma~\ref{lem:6.12} the functor $\tp$ is bijective on morphisms
  and injective on objects.  By Lemma~\ref{lemma:5.34+} it is also
  surjective on objects.  Therefore, $\tp$ is an isomorphism of
  categories.
\end{proof}

To prove Lemma~\ref{lemmaB} it will be convenient for us to define
principal diffeological bundles over manifolds. Note that our definition is 
 different from the one in \cite{IZ}.
\begin{definition}\label{principalbund}
  Let $\calG$ be a diffeological group.  A diffeological space $\scrP$
  is a {\sf principal $\calG$-bundle over a manifold $M$} if the following
  three conditions hold:
\begin{enumerate}
\item there is a surjective map $\varpi:\scrP\to M$ which has local
  sections: for any point $x\in M$, there is a neighborhood $U$ of $M$
  and a smooth section $\sigma:U\to \scrP$ of $\varpi$;
 \item there is smooth right action of $\calG$ on $\scrP$;
 \item the map $\psi: \scrP\times \calG\to \scrP\times_{\varpi,
     M,\varpi}\scrP$ given by $\psi (z, g) :=  (z, z\cdot g)$ is an
   isomorphism of diffeological spaces.
\end{enumerate}
\end{definition}
\begin{lemma}
  Let $M$ be a connected manifold.  Then for any point $x\in M$ the fiber
  $\fs\inv (x)$ of the source map $\fs:\calP/_{\sim}\to M$ of the
  groupoid $\pithin(M)$ with the subspace diffeology
  is a principal $\fund(M,x)$ bundle over $M$; the 
  projection $\varpi: \fs\inv (x) \to M$ is the restriction of the
  target map $\ft:\calP/_{\sim}\to M$: 
\[
\varpi = \ft|_{\fs
    \inv(x)}.
\]
\end{lemma}

\begin{proof}
  Since the multiplication $\fm$ in the fundamental groupoid
  $\pithin(M)$ is smooth, the map
\[
\psi: \fs\inv(x) \times \fund(M,x) \to \fs\inv (x)\times _{\varpi, M, \varpi} \fs\inv (x), \qquad \psi([\tau], [\gamma]):= ([\tau],[\tau] [\gamma])
\]
is smooth.  If $[\tau_1], [\tau_2] \in \fs\inv (x)$ have the same
target, then $[\tau_1]\inv [\tau_2]$ is a loop in $\fund(M,x)$ and
$[\tau_2] = [\tau_1]([\tau_1]\inv [\tau_2])$.   Hence the map 
\[
([\tau_1], [\tau_2])\mapsto ([\tau_1], [\tau_1]\inv [\tau_2])
\]
is a smooth inverse of $\psi$.

By Lemma~\ref{lem:loc_sections}, the map $\varpi =
\ft|_{\fs\inv (x)}:\fs\inv (x) \to M$ has local sections.
\end{proof}

\begin{lemma}
Let $\calG \to \scrP \xrightarrow{\varpi} M$ be a diffeological
principal bundle.  Suppose the diffeological group $\calG$ acts
smoothly on the left on a manifold $T$.   Then the associated bundle 
\[
\scrP\times ^\calG T: = (\scrP\times T)/\calG
\]
is a manifold; the left action of $\calG $ on $\scrP\times T$ is
given by $\gamma \cdot (z, t) = (z\cdot \gamma\inv, \gamma \cdot t)$.
Moreover, $\pi:\scrP\times ^\calG T \to M$, given by $\pi ([z,t]) =
\varpi (z)$, makes $\scrP\times ^\calG T$ into a fiber bundle over $M$
with typical fiber $T$.  Additionally, if $T$ is a $G$-torsor for a Lie group $G$ and $\calG$
acts on $T$ by torsor automorphisms, then $\pi: \scrP\times ^\calG
T\to M$ is a principal $G$-bundle.
\end{lemma}
\begin{proof}
 By Definition~\ref{principalbund} the map
  $\psi: \scrP\times \calG\to \scrP\times_{\varpi, M,\varpi}\scrP$
  given by $\psi (z, g) := (z, z\cdot g)$ is an isomorphism of
  diffeological spaces.  By composing its inverse $\psi\inv$ with the
  projection on the second factor, we obtain a smooth map 
\[
d: \scrP \times_M \scrP \to \calG
\]
characterized by 
\[
z_1 \cdot d(z_1,z_2) = z_2
\]
for all points $(z_1, z_2)$ in the fiber product $\scrP \times_M \scrP $.
Consequently, if $s_\alpha:U_\alpha \to \scrP$ and $s_\beta:U_\beta \to
\scrP$ are two local sections, then 
\[
s_\alpha(x) \cdot d (s_\alpha(x), s_\beta (x)) = s_\beta (x)
\]
for all $x\in U_\alpha \cap U_\beta$.
A local section $s_\alpha:U_\alpha \to \scrP$ defines a smooth map 
\[
\sigma_\alpha: U_\alpha \times T \to (\scrP\times
^\calG T)|_{U_\alpha}, 
\qquad  \sigma_\alpha (x,t) := [s_\alpha(x), t]. 
\]
By construction, the diagram
\begin{equation} \label{eq:6.111}
\xy
(-15,6)*+{U_\alpha \times T}="1";
(15,6)*+{(\scrP\times^\calG T)|_{U_\alpha}}="2";
(0,-10)*+{U_\alpha}="3";
 {\ar@{->}^{\sigma_\alpha\ \ \ \ } "1";"2"};
 {\ar@{->}_{pr_1} "1";"3"}; 
{\ar@{->}^{\pi} "2";"3"};
\endxy
\end{equation}
commutes, where $pr_1$ is the projection on the first factor and
$\pi([z,t]) := \varpi(z)$.   

 The map $\sigma_\alpha$ has a smooth
inverse: given $x\in U_\alpha$ and $z\in \scrP_x := \varpi\inv (x)$, we have 
\[
s_\alpha (x) \cdot d(s_\alpha (x), z) = z.
\]
Hence, for all $t\in T$, 
\[
[z,t] = [s_\alpha (x) \cdot d(s_\alpha (x), z) , t] = [s_\alpha (x),
d(s_\alpha (x), z) \cdot  t].
\]
Consequently 
\[
\sigma_\alpha \inv ([z,t]) = (\varpi(z), d(s_\alpha (\varpi(z))\cdot t)
\]
is smooth.  We conclude that $\sigma_\alpha :U_\alpha \times T
\to (\scrP \times^\calG T)|_{U_\alpha}$ is an isomorphism of
diffeological spaces.   

Now choose an open cover $\{U_\alpha\}$ of the manifold $M$ so that
for each index $\alpha$, the restriction $\scrP|_{U_\alpha } \to
U_\alpha$ has a section $s_\alpha$.  Then the images of the
corresponding trivializations $\sigma_\alpha:U_\alpha \times T\to
\scrP\times^\calG T$ cover $\scrP\times ^\calG T$.  Any diffeological
space that has an open cover consisting of manifolds is itself a
manifold \cite[Section~4.2, p.~78]{IZ}. Therefore the diffeological
space $\scrP\times ^\calG T$ is a manifold.  Moreover, since diagram
\eqref{eq:6.111} commutes, $\pi:\scrP\times ^\calG T\to M$ is a
locally trivial fiber bundle with typical fiber $T$.

Additionally, if $T$ is a $G$-torsor for a Lie group $G$ and $\calG$
acts on $T$ by torsor automorphisms, then $\scrP\times ^\calG T$
admits a right $G$ action.  By construction, the local trivialization
maps $\sigma_\alpha: U_\alpha \times T \to \scrP\times ^\calG T$ are
$G$-equivariant.  It follows that $\pi:\scrP\times ^\calG T \to M$ is
a principal $G$-bundle.
\end{proof}

\begin{proof}[Proof of Lemma~\ref{lemmaB}]
Let $\alpha:F\Rightarrow F'$ 
be a natural isomorphism between  two transport functors.  Fix a point $x\in M$. Then $\alpha_x:F(x)\to F'(x)$ is a map of $G$-torsors.  Hence, it defines a smooth, $(\fund(M,x)\times G)$-equivariant map: 
\[
(id, \alpha_x):\fs\inv (x) \times F(x) \longrightarrow \fs\inv (x) \times F'(x).
\]

This map descends to a smooth $G$-equivariant map on the quotient
\[[id,\alpha_x]: (\fs\inv (x) \times F(x))/\fund(M,x) \longrightarrow
(\fs\inv (x) \times F'(x))/\fund(M,x).\] There is also a natural
action of $\pithin(M)$ on $\fs\inv(x)$ that descends to an action on
the quotient.  Since $[id, \alpha_x]$ is $\pithin(M)$-equivariant the
procedure defines a functor
\[
\assoc: \trans(M) \longrightarrow B^p G(M).
\]
For any principal $G$-bundle $P\to M$ with an action $a$ of $\pithin
(M)$, the principal $G$-bundle $\assoc (\rep(a))$ is naturally
isomorphic to the bundle $P$.  The isomorphism 
\[
\eta_P:(\fs\inv (x)\times P_x)/\fund(M,x) \longrightarrow P
\]
is defined by 
\[
  \eta_P(([[\gamma], z])):= [\gamma]\cdot z.
\]
Conversely, any transport functor $F:\pithin(M)\to \gtors$ is isomorphic
to the transport functor $\rep(\assoc(F))$.  To see this, note first
that $\rep(\assoc(F))$ is defined on objects by sending $y\in M$ to
the fiber of the bundle $\assoc(F)$ above $y$.  This fiber is the
torsor 
\[\left\{[[\gamma], z] \in \fs\inv (x)\times F(x)/\calG \mid \gamma
(1) = y\right\}.\] 
The natural isomorphism $\varepsilon:
\rep(\assoc(F))\Rightarrow F $ is given by
\[
\varepsilon_y: [[\gamma], z]\mapsto F([\gamma])z.
\]
It is well-defined.  It follows that the functor $\rep: B^p G(M)\to
\trans(M)$ is an equivalence of categories.
\end{proof}


\part{Parallel transport and stacks}

From now on we assume that the reader is familiar with stacks over the
site of differentiable manifolds.  The standard references are Behrend
and Xu \cite{BX}, Heinloth \cite{H} and Metzler \cite{Metzler}.  We
will primarily think of stack $\mathcal{X}$ over the site $\Man$ of
manifolds as a category fibered in groupoids (CFG) that satisfies
descent.  One can also think of stacks over $\Man$ as lax presheaves
of groupoids with descent.  Grothendieck construction (see for example
\cite{V}) converts lax presheaves into CFGs.  A choice of cleavage
turns a CFG into a lax presheaf.  Finally recall that any Lie groupoid
$\Gamma =\{\Gamma_1\toto\Gamma_0\}$ has a stack quotient
$[\Gamma_0/\Gamma_1]$: it is a category fibered in groupoids over
$\Man$ whose objects are principal $\Gamma$-bundles (see \cite{L} for
example). It is well known that stack quotients, as the name implies,
are stacks.

\section{Holonomy functor as an isomorphism of stacks }
We start by constructing a presheaf of groupoids out of the assignment
of transport functors to manifolds.
\begin{lemma}\label{lem:5.1}
  The assignment
\[
M\mapsto \trans(M)
\]
extends to a contravariant functor, that is, a strict presheaf of
groupoids
\[
\trans: \Man^{op} \to \gpd
\]
from the category of manifolds to the category of groupoids.
\end{lemma}

\begin{proof}
We need to check that for any smooth map $f:N\to M$ between manifolds,
we have a map $f^*: \trans(M) \to \trans(N)$ such that, for a pair of
composable maps $Q\xrightarrow{h}N\xrightarrow{f}M$, we have $(f\circ
h)^* = h^*\circ f^*$ and that $id_M^* = id_{\trans{M}}$.  

Since $\pithin$ is a functor from the category of manifolds to the
category of diffeological groupoids, for any smooth map $f:N\to M$
between manifolds, we have a morphism $\pithin(f):\pithin(N)\to
\pithin(M)$ between diffeological groupoids (see
Proposition~\ref{prop:3.27}).  In particular $\pithin(f):
\calP(N)/_{\sim}\to \calP(M)/_{\sim}$ is a map of diffeological spaces
and so is the restriction $\pithin(f))|_{\fund(N,x)}$ for any point
$x\in N$.  If $T:\pithin(M)\to \gtors$ is a transport functor then, by
definition, for any $x\in N$ the map
$T|_{\fund(M,f(x))}:{\fund(M,f(x))} \to \Aut(T(f(x)))$ is smooth.
Consequently the map
\[ (T\circ \pithin(f))|_{\fund(N,x)}= T|_{\fund(M,f(x))} \circ
(\pithin(f)|_{\fund(N,x)})
\] is smooth as well.  Thus,
\[ 
f^*T := T\circ \pithin(f)
\] 
is a transport functor.  Since $\pithin$ is a functor,
\[ (f\circ h)^*T = T \circ \pithin(fh) = T\circ \pithin(f) \circ
\pithin(h) =h^*(f^*T).
\] 
Since $\pithin(id_M) = id_{\pithin(M)}$, $(id_M)^* = id_{\trans{(M)}}$
for all manifolds $M$.
\end{proof}

Grothendieck construction \cite{V} applied to the presheaf of groupoids
  $\trans$ produces a category $\transu$ which is fibered in groupoids over the
  category of manifolds $\Man$.  Explicitly we define $\transu$ as follows.

  \begin{definition}[the category $\transu$ of transport functors over
    the category $\Man$ of manifolds]\label{def:transu}
  The objects of $\transu$ are pairs $(M, F)$ where $M$ is
  a manifold and $F\in \trans(M)$ is a transport functor.  A morphism
  of $\transu$ from $(N,H)$ to $(M,F)$ is a pair $(f,\alpha)$ where
  $f:N\to M$ is a smooth map of manifolds and $\alpha:F\circ
  \pithin(f) \Rightarrow H$ is a natural isomorphism.

The functor $\varpi_T:\transu\to \Man$ is given on arrows by
\[
\varpi_T ((N,H)\xrightarrow{(f,\alpha)} (M,F)) = (N\xrightarrow{f}M).
\]
\end{definition}
We would like to extend Theorem~\ref{thm:sub_main} to a statement about maps of stacks. As a first step we prove
\begin{lemma} \label{lem:5.166}
The collection of functors
\[
\{ \hol_M: B^\nabla G(M)\to \trans(M)\}_{M\in \Man}
\]
extends to a functor
\[
\hol: B^\nabla G \to \transu ,
\]
which is a  morphism of categories fibered in groupoids over $\Man$.
\end{lemma}

\begin{proof}
Let $f:(P\to M, A) \to (P'\to M', A')$ be a map of principal
$G$-bundles with connections.  The connections $A$ and $A'$ define transport functors
$\hol(P,A) :\pithin(M)\to \gtors$ and $\hol(P',A') :\pithin(M')\to
\gtors$, respectively.  We need to define a morphism
\[
\hol(f): \hol(P,A) \to \hol(P',A') 
\]
in $\transu$. Such a morphism is a pair of the form $(\bar{f}, \eta)$
where $\eta: \hol(P,A) \Rightarrow \hol (P', A') \circ \pithin(\bar
f)$ is a natural transformation (see Definition~\ref{def:transu}
above) and $\bar{f}:M\to M'$ the induced map on the base.

Recall that the objects of the groupoid $\pithin(M)$ are points of
$M$.  For each point $x\in M$, define
\[
\eta_x := f|_{P_x}: P_x \to P'_{\bar{f}(x)}.
\]
We check that the collection $\{\eta_x\}_{x\in M}$ is a
natural transformation. 
Since \eqref{eq:5.1} holds, the diagram 
\begin{equation}\label{eq:5.3}
\xy
(-22,10)*+{P_x}="1";
(22,10)*+{P_y}="2";
(-22,-6)*+{P'_{\bar{f}(x)}}="3";
(22,-6)*+{P'_{\bar{f}(y)}}="4";
 {\ar@{->}^{||_{\gamma}} "1";"2"};
 {\ar@{->}^{f|_{P_y}} "2";"4"};
 {\ar@{->}_{||_{\bar{f}\circ\gamma}} "3";"4"};
{\ar@{->}_{f|_{P_x}} "1";"3"};
\endxy
\end{equation}
commutes.
By definition, $\hol(P,A) ([\gamma]) = ||_\gamma$, the parallel
transport along $\gamma$ in $P$ defined by the connection $A$. On the
other hand $\pithin(\bar{f}) ([\gamma])=[\bar{f}\circ \gamma]$ and
$\hol(P',A') ([\bar{f}\circ \gamma]) = ||_{\bar{f}\circ
  \gamma}$.  Therefore 
the diagram
\begin{equation}\label{eq:5.2}
\xy
(-22,10)*+{P_x}="1";
(22,10)*+{P_y}="2";
(-22,-6)*+{P'_{\bar{f}(x)}}="3";
(22,-6)*+{P'_{\bar{f}(y)}}="4";
 {\ar@{->}^{\hol(P,A)([\gamma])} "1";"2"};
 {\ar@{->}^{\eta_y} "2";"4"};
 {\ar@{->}_{\hol(P',A') \circ (\pithin(\bar{f}))([\gamma])} "3";"4"};
{\ar@{->}_{\eta_x} "1";"3"};
\endxy
\end{equation}
commutes for every arrow $x\xrightarrow{[\gamma]}y$ in
$\pithin(M)$.   Therefore $\eta$ {\em is} a natural
transformation.  Thus $\hol(f)$ is a morphism in
$\transu$.

It is not hard to check that $\hol$ 
is actually a
functor. Finally the functor $\hol$ commutes with the projections
$\varpi_B:B^\nabla G\to \Man$, $\varpi_T: \trans \to \Man$ to the
category of manifolds since 
\[
\varpi _B (f:(P\to M, A) \to (P'\to M', A')) =
\bar{f} 
\]
and
\[
\varpi_T ( \hol (f:(P\to M, A) \to (P'\to M', A'))) =
\varpi (\bar{f}, \eta) = \bar{f}.
\] 
\end{proof}
We are now in position to state and prove the main result of the
paper. The proof is short since most of the work has already been
done.
\begin{theorem}\label{thm:main} The functor 
\[
\hol: B^\nabla G \to \transu
\]
is an equivalence of categories fibered in groupoids over $\Man$.
\end{theorem}

\begin{proof}
  Recall that a functor between two categories fibered in groupoids is
  an equivalence of categories if and only if its restriction to each
  fiber is an equivalence of categories; see for example
  \cite[Proposition~3.36]{V}.  By Theorem~\ref{thm:sub_main} for each
  manifold $M$ the functor $\hol_M:B^\nabla G \to \transu(M)=
  \trans(M)$ is an equivalence of categories.
\end{proof}
As an immediate corollary of Theorem~\ref{thm:main}, we obtain:
\begin{corollary} \label{cor:5.5}
   The category $\transu$ of transport functors is a stack over the
  category (site) $\Man$ of manifolds.
 \end{corollary}
\begin{proof}
  Since the CFG $\transu\to \Man$ is equivalent to the CFG $ B^\nabla
  G \to \Man$ and since $B^\nabla G$ is a stack, $\transu$ is a stack.
\end{proof}

\section{Principal bundles over stacks and parallel transport}
\label{sec:6}

In this section we work out some consequences of Theorem~\ref{thm:main}
for principal bundles with connections over stacks.  We start with a
definition of a principal $G$ bundle over a stack (where as before $G$
is a Lie group) which is known to experts \cite{ BMW, BN}.  We
then define principal bundles with connections over stacks and the
associated parallel transport.  We prove that for each stack $\calX$
the functor $\hol$ induces an equivalence between the category of
principal $G$-bundles with connections over $\calX$ and the category
of corresponding transport functors.
\begin{definition}
  Let $G$ be a Lie group.  A {\sf principal $G$-bundle over a stack
    $\calX \to \Man$} is a a 1-morphism of stacks $p:\calX \to BG$,
  where $BG$ denotes the stack of principal $G$ bundles.  
\end{definition}
There several reasons why this definition makes sense.

\begin{itemize}
\item If the stack $\calX$ is a manifold $M$ then by 2-Yoneda (see
  \cite{V}, for example) there is an equivalence of categories $[M,
  BG]\xrightarrow{\simeq} BG(M)$.  Under this equivalence a functor
  $p\in [M,BG]$ corresponds to the principal bundle $p(id_M)$ over
  $M$.  Thus functors $p:M\to BG$ ``are'' principal $G$-bundles over
  $M$.
\item Suppose the stack $\calX$ is a stack quotient $[\Gamma_0/\Gamma_1]$
  of a Lie groupoid $\Gamma$. The 
  bicategory $\Bi$ of Lie groupoids, bibundles and bibundle
  isomorphism is 2-equivalent to the 2-category of geometric stacks over $\Man$ (see \cite{L} or \cite{Blo}).   Consequently the functor
  category $[ [\Gamma_0/\Gamma_1], BG]$ is equivalent to the category of
  bibundles from $\Gamma$ to the action groupoid $\{G\toto *\}$:
  \[ [ [\Gamma_0/\Gamma_1], BG] \xrightarrow{\simeq}
\left\{ P:\{\Gamma_1\toto \Gamma_0\} \to \{G\toto *\} \mid 
P \textrm{ is a right $G$ principal bibundle }
\right\}.
\]
Any bibundle $P:\{\Gamma_1\toto \Gamma_0\} \to \{G\toto *\}$ is a
principal $G$ bundle over the Lie groupoid $\Gamma$ (see \cite{GTX}).
\item The functor category $[ [\Gamma_0/\Gamma_1], BG]$ is also
  equivalent to the cocycle category $BG (\Gamma_1\toto \Gamma_0)$
  (see Definition~\ref{def:X-Gamma} and Proposition~\ref{prop:Noohi}
  below).  Objects of the cocycle category $BG (\Gamma_1\toto
  \Gamma_0)$ are again principal $G$-bundles over the groupoid
  $\Gamma$.  In particular, if the groupoid $\Gamma$ is a cover
  groupoid arising from a cover $\{U_\alpha\}$ of a manifold $M$ then
  the objects of the cocycle category are \v{C}ech cocycles with values in
  the Lie group $G$.  Hence they ``are'' principal $G$-bundles over the
  manifold $M$.
\end{itemize}

By analogy with the notion of a principal $G$ bundle over a stack
$\calX$ we define principal bundles with connections and parallel
transport functors over $\calX$ as follows.

\begin{definition} Let $\calX$ be a stack over $\Man$ and $G$ a Lie
  group.  We define the {\sf category of principal $G$ bundles with
    connection over the stack $\calX$} to be the functor category
  $[\calX ,B^\nabla G]$.  

  In particular a {\sf principal bundle with connection over a stack}
  $\calX$ is a 1-morphism of stacks $F:\calX \to B^\nabla G$.
\end{definition}
\begin{definition} Let $\calX$ be a stack over $\Man$ and $G$ a Lie
  group.  We define the {\sf category of parallel transport functors
    over the stack $\calX$} to be the functor category $[\calX
  ,\transu]$.

  In particular a {\sf parallel transport functor on a stack }
  $\calX$ is a 1-morphism of stacks $T:\calX \to \transu$.
\end{definition}
We have the following extension of
Theorem~\ref{thm:sub_main} from manifolds to stacks.
\begin{theorem}\label{thm:6.4}
  For any stack $\calX$ over the category of manifold the functor
  $\hol:B^\nabla G\to \transu$ induces an equivalence of categories
\[
\hol_*: [\calX ,B^\nabla G]\to [\calX ,\transu],
\]
where the functor $\hol_*$ is defined  by 
\[
\hol_*(\alpha: F\Rightarrow H):=  
(\hol\circ \alpha:\hol\circ F \Rightarrow \hol\circ H)
\]
for a morphism $(\alpha: F\Rightarrow H)\in [\calX ,B^\nabla G]$.
\end{theorem}

\begin{proof}
Since $\hol$ is an equivalence of categories so is $\hol_*$.
\end{proof}
We now interpret the results of the above theorem more concretely in
terms of the cocycle category.  Recall that a Lie groupoid $\Gamma =
\{\Gamma_1 \toto \Gamma_2\}$ gives rise to a simplicial manifold.  In
particular we have three face maps
\[
d_i:\Gamma_1 \times _{\Gamma_0}
\Gamma_1 \to \Gamma_1, \qquad i=0,1,2
\]
which are defined by 
\[
 d_0(\gamma_1,\gamma_2) = \gamma_2,\qquad
d_1(\gamma_1,\gamma_2) = \gamma_1 \gamma_2 \qquad \textrm{ and }\qquad
d_2(\gamma_1,\gamma_2) = \gamma_1.
\]
\begin{definition}[The category $\calX(\Gamma)$  of $\Gamma$ cocycles
  ]\label{def:X-Gamma}\mbox{}
  Let $\Gamma = \{\Gamma_1\toto \Gamma_0\}$ be a Lie groupoid and
  $\calX\to \Man$ a category fibered in groupoids.  We define {\sf the
    category $\calX(\Gamma)$ of $\Gamma$ cocycles with values in $\calX$} as
  follows. The objects of $\calX(\Gamma)$ are pairs $(p,\varphi)$
  where $p$ is an object of $\calX(\Gamma_0)$ and $\varphi: s^*p\to
  t^*p$ is an arrow in $\calX(\Gamma_1)$. The morphism $\varphi$ is subject  to the cocycle
  condition
\[
d_2^*\varphi \circ d_0^* \varphi = d_1 ^* \varphi,
\]
where $d_i:\Gamma_1 \times _{\Gamma_0} \Gamma_1 \to \Gamma_1$ are the
face maps defined above.  A morphism from $(p, \varphi)$ to $(p',
\varphi')$ is a morphism $\alpha:p\to p'$ in $\calX(\Gamma_0)$ such
that the diagram
\[
\xy
(-10,10)*+{ t^*p}="1";
(20,10)*+{s^*p}="2";
(-10,-7)*+{t^*p'}="3";
(20,-7)*+{s^*p'}="4";
 {\ar@{->}_{t^*\alpha} "1";"3"};
 {\ar@{->}^{s^*\alpha} "2";"4"};
 {\ar@{->}^{\varphi} "1";"2"};
{\ar@{->}_{\varphi'} "3";"4"};
\endxy
\]
commutes.
\end{definition}
The following fact is well-known to experts.
\begin{proposition}\label{prop:Noohi}
Let $\Gamma$ be a Lie groupoid and $\calX \to \Man$ a category fibered in groupoids.
There is a canonical functor  
\[
\Sigma: [[\Gamma_0/\Gamma_1], \calX] \to \calX (\Gamma)
\]
from the functor category $ [[\Gamma_0/\Gamma_1], \calX]$ to the
cocycle category $\calX (\Gamma)$.

If moreover $\calX$ is
a stack then $\Sigma$ is an equivalence of categories (i.e., an
isomorphism of stacks).
\end{proposition}
\begin{proof}
  We follow the custom of identifying a manifold $M$ with the stack
  $\Hom(\cdot, M)$ without further comment.  Recall that the stack
  quotient $[\Gamma_0/\Gamma_1]$ is a geometric stack.  The canonical
  atlas $p:\Gamma_0 \to [\Gamma_0/\Gamma_1]$ is characterize by the
  fact that $p(id_{\Gamma_0})$ is the principal $\Gamma$ bundle
  $t:\Gamma_1\to \Gamma_0$ (see \cite{L}).  Recall further that the
  manifold  $\Gamma_1$ is the 2-categorical fiber product
  $\Gamma_0\times_{p,[\Gamma_0/\Gamma_1],p}\Gamma_0$ and the diagram
\[
\xy
(-8,8)*+{\Gamma_1 }="1"; 
(-8,-8)*+{\Gamma_0}="2"; 
(8,8)*+{\Gamma_0}="3"; 
(8,-8)*+{[\Gamma_0/\Gamma_1]}="4"; 
{\ar@{->}_{s } "1";"2"};
{\ar@{->}^{t} "1";"3"};
{\ar@{->}_{p} "2";"4"}; 
{\ar@{->}^p "3";"4"};
{\ar@{=>}^<<<{\beta} (-2,-2)*{};(1,1)*{}} ;
\endxy
\]
2-commutes. Also
\begin{equation}
(\beta\circ d_2)*(\beta\circ d_0) = \beta \circ d_1
\end{equation}
as natural isomorphisms from $p\circ s \circ d_0 =p\circ s \circ d_1$
to $p\circ t \circ d_2 = p\circ t \circ d_1$ (here $*$ denotes the vertical composition of natural transformations).  Consequently for any
CFG $\calX \to \Man$ and any 1-morphism $f:[\Gamma_0/\Gamma_1] \to
\calX$ of CFGs we have a natural isomorphism
\[
f\circ \beta: f\circ p \circ s \Rightarrow f\circ p \circ t
\]
satisfying 
\begin{equation}\label{eq:6.2}
(f\circ \beta\circ d_2)*(f\circ \beta\circ d_0) = f\circ \beta \circ d_1.
\end{equation}
Consider now the object $P: = (f\circ p)\,(id_{\Gamma_0})\in \calX
(\Gamma_0)$.  For any map $h:\Gamma_1\to \Gamma_0$
\[
(f\circ p\circ h)\, (id_{\Gamma_1}) = h^* ( (f\circ p)\,(id_{\Gamma_0}))
\]
Consequently 
\[
\varphi:= (f\circ \beta)\,(id_{\Gamma_1})
\]
is an isomorphism in $\calX (\Gamma_1)$ from $s^*P$ to $t^*P$.  Equation \eqref{eq:6.2} translates then into the cocycle condition 
\[
(d_2^*\varphi)\circ (d_0^*\varphi) = d_1^*\varphi.
\]
Therefore the pair $(P= (f\circ p)\,(id_{\Gamma_0}), \varphi = (f\circ
\beta)\,(id_{\Gamma_1}))$ is an object of the cocycle category
$\calX(\Gamma)$.  This defines the functor $\Sigma$ on objects.
Similarly given a morphism $\gamma:f\Rightarrow h$ in the functor
category $[[\Gamma_0/\Gamma_1], \calX]$ we get a morphism
\[
\alpha = (\gamma \circ p)\,(id_{\Gamma_0}): (f\circ
p)\,(id_{\Gamma_0}) \to (h\circ p)\,(id_{\Gamma_0})
\]
in $\calX(\Gamma_0)$.  It is not hard to check $\alpha$ is  a
morphism in $\calX(\Gamma)$ from $\Sigma(f)$ to $\Sigma(h)$.  This
defines the functor $\Sigma$ on morphisms.

A proof that $\Sigma $ is an equivalence of categories if $\calX$ is a
stack is a bit more involved.  We refer an interested reader to
\cite{Noohi}[Proposition~3.19].
\end{proof}

We are now in position to reformulate Theorem~\ref{thm:6.4} for geometric
stacks in terms of cocycles.

\begin{theorem}\label{thm:6.7}
  For any Lie groupoid $\Gamma$ the functor $\hol$ induces an
  equivalence of categories
\[
\hol_\Gamma: B^\nabla G (\Gamma)\to \transu(\Gamma).
\]
\end{theorem}

\begin{proof}
  By Theorem~\ref{thm:6.4} the functor categories $[[\Gamma0/\Gamma_1]
  ,B^\nabla G]$ and $[[\Gamma0/\Gamma_1] ,\transu]$ are equivalent.
  By Proposition~\ref{prop:Noohi} the first functor category is
  equivalent to the cocycle category $B^\nabla G (\Gamma)$ and the
  second functor category is equivalent to the cocycle category
  $\transu(\Gamma)$.
\end{proof}

Alternatively one can view Theorem~\ref{thm:6.7} as an instant of the following general fact: 
\begin{proposition}\label{prop:last}
  Suppose $\Gamma$ is a Lie groupoid, $\calX,\calY\to \Man$ are two
  categories fibered in groupoids and $F:\calX \to \calY$ is a
  1-morphism of fibered categories. Then the functor $F$ induces a
  functor
\[
F_\Gamma:\calX(\Gamma)\to \calY(\Gamma)
\]
between cocycle categories. Moreover if $F$ is an equivalence of
categories then so is $F_\Gamma$.
\end{proposition}
We discuss a proof of   Proposition~\ref{prop:last} in Appendix~\ref{app:2}.

\begin{appendix} 
\part{Appendices}
  \section{Diffeological spaces and diffeological
    groupoids}\label{app:1}

  The goal of this section is to recall two definitions and some
  properties of the category $\DS$ of diffeological spaces and to
  prove the folklore result that the thin fundamental groupoid is a
  groupoid internal to the category of diffeological spaces.  We start
  by recalling the ``traditional'' definition of a diffeological
  space, which is due to Souriau \cite{S}.  A similar notion was
  independently introduced by K.-T.\ Chen \cite{Chen1,Chen2}.  Our
  primary references are \cite{IZ} and \cite{BH}.
\begin{definition}[Diffeology]\label{def:diffeol1}
  A {\sf diffeology} on a set $X$ is a collection of functions 
  \[D_X\subset\{f:U\rightarrow X\ |\ U\subset \R^n\ \textrm{open},\
  n\in\mathbb{N}\}\] satisfying the following three conditions:
  \begin{enumerate}
\item  every
  constant map is in $D_X$;
\item if $V\subset\R^m$ is open, $f:U\to X$ is in $D_X$, and $g:V\to
  U$ is a smooth map, then the composite $f\circ g: V\to X$ is also in
  $D_X$;
\item if $\{U_i\}$ is an open cover of $U\subset \R^n$ and $f:
  U\to X$ is a map of sets such that
 $f|_{U_i}\to X$ is in $D_X$, then $f:U\to X$ is in $D_X.$
\end{enumerate}

The pair $(X,D_X)$ is called a {\sf diffeological space}, and the
elements of $D_X$ are called {\sf plots}.
\end{definition}

\begin{remark}
  Just as in the case of  topological spaces, it is common to refer to a
  diffeological space $(X, D_X)$ simply as $X$ with the choice of a
  diffeology $D_X$ suppressed from the notation.
\end{remark}

\begin{example}
  Any smooth manifold $M$ is a diffeological space.  The set in
  question is the underlying set of the manifold $M$, and the
  collection of plots $D_M$ consists of all smooth maps from all open
  subsets $U $ of $ \bigsqcup_{n=0}^\infty \R^n$ to the manifold $M$.
\end{example}

\begin{definition}\label{def:mapofdiffeos}
  A {\sf map of diffeological spaces} or a {\sf smooth map} from a
  diffeological space $(X,D_X)$ to a space $ (Y,D_Y)$ is a map of
  sets $f:X \to Y$ such that for any plot $p:U \to X$ in $D_X$, the
  composite $f\circ p$ is in $D_Y$.
\end{definition}

\begin{remark}[The category $\DS$ of diffeological spaces]
  The composite of two smooth maps between diffeological spaces is
  smooth. Thus, diffeological spaces and smooth maps form a category
  which we denote by $\DS$.  It is well-known \cite{IZ} that the
  category of manifolds embeds into the category of diffeological
  spaces. That is, a map $f:M\to N$ between two manifolds is smooth in
  the diffeological sense if and only if it is a smooth map of
  manifolds.
\end{remark}

The category $\DS$ has many nice properties \cite{LaubingerDiss,
  BH}. For instance, it has all small limits and colimits. Also the
space of maps between two diffeological spaces is again naturally a
diffeological space.
It will be useful for us to write down explicitly several
corresponding constructions.  We start with a definition.

\begin{definition}[Subspace diffeology] Let $(X, D_X)$ be a
  diffeological space and $Y\subset X$ a subset.  The
  {\sf subspace diffeology} $D_Y$ is the set:
\[
D_Y:= \{(p:U\to X) \in D_X\mid p(U) \subset Y\}
\]
\end{definition}
To introduce the quotient diffeology it is convenient to switch
our point of view and think of diffeological spaces as certain kinds
of sheaves of sets.

\begin{definition}[The category $\mathsf{Open}$] The objects of the
  category $\mathsf{Open}$ are by definition all open subsets of all
  coordinate vector spaces $\R^n$, $n\geq 0$, or, equivalently, open
  subsets of $\bigsqcup_{n=0}^\infty \R^n$. A morphism in
  $\mathsf{Open}$ from an open set $U$ to an open set $V$ is a smooth
  map $f:U\to V$.
\end{definition}
\begin{definition}
  A (set-valued) presheaf $\CR$ on a category $\Open$ is a contravariant
  functor from $\Open$ to the category $\Set$ of sets:
\[
\CR:\Open^{op}\to \Set.
\]  
\end{definition}
\begin{notation}
  Given a presheaf $\CR:\Open^{op}\to \Set$ and an arrow $f:U\to V$ in
  $\Open$, we get a map of sets $\CR(f): \CR(V)\to \CR(U)$, which we
  think of as a pullback along $f$.  Thus, given an element $s\in
  \CR(V),$ we denote $\CR(f) s \in \CR(U)$ by $f^*s.$ If $U\subset V$
  and $f$ is the inclusion, we may also write $s|_U$ for $f^*s$.
\end{notation}
\begin{definition}
  A presheaf $\CS:\Open^{op}\to \Set$ is a {\sf sheaf} if for any open set
  $U\in \Open$, any open cover $\{U_i\}_{i\in I}$ of $U$, and
  any collection of elements $s_i\in \CS (U_i)$ with 
\[
s_i|_{U_i\cap    U_j} = s_j |_{U_i\cap U_j},
\] 
there exists a unique element $s\in \CS(U)$ with
\[
s|_{U_i} = s_i.
\]
\end{definition} 
\begin{example} \label{ex:underline{X}}
Any set $X$ defines a sheaf $\underline{X}$ on $\Open$:
\[
\underline{X}(U\xrightarrow{f}V): = Map(U,X)\xleftarrow{f^*}Map(V,X), 
\]
where $Map(U,X)$ denotes the set of maps from $U$ to $X$ and $f^*$
denotes the pullback by $f$.
Note that for the one element set $*:=\R^0 \in \Open$ the evaluation map
\[
\underline{X}(*) = Map(*,X)\xrightarrow{\simeq} X
\]
is a bijection of sets.
\end{example}

\begin{definition}
  A  presheaf $\CR:\Open^{op}\to \Set$ is {\sf subpresheaf} of a presheaf
$\CS:\Open^{op}\to \Set$ if for every open set $U\in \Open$
\[
\CR(U)\subset \CS(U).
\] 
 \end{definition}
 We are now in position to define diffeologies in terms of sheaves.
\begin{definition}[Diffeology, as a sheaf]\label{def:diffeol2} 
  Let $*$ be the one element set $\R^0.$ A {\sf diffeology} $D=D_X$ on
  a set $X$ is a subsheaf of the sheaf $\underline{X}$, defined in
  Example~\ref{ex:underline{X}}, such that $D(*) = \underline{X}(*)$.
\end{definition}

\begin{remark}
  If $D$ is a diffeology on a set $X$ in the sense of
  Definition~\ref{def:diffeol2}, then $D(U)\subset \underline{X}(U)
  \equiv Map(U, X)$ for any $U\in \Open$.  It is for this reason that
  elements of $D(U)$ may be thought of as the plots of
  Definition~\ref{def:diffeol1}.
\end{remark}

Since the category of diffeological spaces has all small colimits
(see \cite{BH}), it has quotients.  Explicitly they can
be constructed as follows.
\begin{construction}[Quotient diffeology]\label{constr:quot-diffeo}
  Let $(X, D_X)$ be a diffeological space, $R\subset X\times X$ an
  equivalence relation, $Y =X/R$ the set of equivalence classes and
  $q:X\to Y$ the quotient map.  The {\sf quotient diffeology} $D_Y$ on
  $Y$ is constructed as the sheafification of the presheaf
  $D_{pre}(Y)$ defined by
\[
D_{pre}(Y):= \{ p:U\to Y\mid \textrm{ there is } (\tilde{p}:U\to X)
\in D_X \textrm{ with } q\circ \tilde{p} = p\}.
\]
Explicitly, for any open $U\subset \R^n,$ a map $p:U\to Y$ is
a plot in the quotient diffeology $D_Y$ if and only if for every
$u\in U$ there is an open neighborhood $V$ of $u$ in $U$ and
$\tilde{p}:V\to X$ with
\[
q\circ \tilde{p} = p|_V.
\]

The quotient map $q:(X,D_X)\to (Y,D_Y)$ has the following universal
property.  Give the product $X\times X$ the product diffeology
$D_X\times D_X$(see Remark~\ref{rmrk:product_diflogy} below) and
$R\subset X\times X$ the subspace diffeology.  Then $X \xrightarrow{q}
Y$ is the coequalizer of the diagram $R\toto X$ in the category $\DS$
of diffeological spaces.
\end{construction}

\begin{lemma} \label{lemma:3.18} Let $q:X\to Y$ be a quotient map
  between two diffeological spaces (that is, $q$ is the coequalizer of
  a diagram $R\toto X$ in $\DS$ for some equivalence relation $R$ on
  $X$), and let $Z$ be another diffeological space. Then a map $f:Y\to
  Z$ is smooth if and only if for any plot $p:U\to X$ the composite
  $f\circ q \circ p:U\to Z$ is a plot on $Z$.
\end{lemma}

\begin{proof}
  If $f$ is smooth, then $q\circ f$ is smooth.  Therefore for any plot
  $p:U\to X$ the composite $f\circ q \circ p:U\to Z$ is a plot on
  $Z$. 

  Conversely suppose that for any plot $p:U\to X$ the composite
  $f\circ q \circ p:U\to Z$ is a plot on $Z$ and suppose $r:U\to Y$ is
  a plot.  By definition of quotient diffeology there is an open cover
  $\{U_\alpha\}$ of $U$ and a collection of plots $r_\alpha :U_\alpha
  \to X$ so that
\[
q\circ r_\alpha = r|_{U_\alpha}.
\]
By assumption,
\[
(f\circ r)|_{U_\alpha} = f\circ (r_{U_\alpha}) = f\circ q \circ
r_\alpha :U_\alpha \to Z
\]
are plots. Since a diffeology is a sheaf, $f\circ r$ is a plot. 
Therefore $f$ is smooth.
\end{proof}

Since the category of diffeological spaces has all small limits, 
it has fiber products.  Explicitly, they can be constructed as follows.

\begin{construction}[Fiber product diffeology]\label{constr:fpdiffeo}
  Let $f:(X,D_X) \to (Z,D_Z)$ and $g:(Y,D_Y) \to (Z,D_Z)$ be two
  maps of diffeological spaces.  We construct their fiber
  product %
  as follows: the underlying set is the fiber product
\[
X\times _{f,Z,g} Y := \{ (x,y)\in X\times Y \mid f(x) = g(y)\},
\]
and the set of plots 
is 
\[
D(X\times_{f,Z,g} Y): = 
\{ (p_X, p_Y) \in D_X\times D_Y \mid f\circ p_X = g\circ p_Y\}.
\]
\hfill It is not hard to check that the diffeological space $(X\times
_{f,Z,g} Y, D(X\times_{f,Z,g} Y))$ together with the obvious maps to
$(X, D_X)$ and $(Y, D_Y)$ is a fiber product in category of
diffeological spaces.  $\Box$
\end{construction}

\begin{remark} \label{rmrk:product_diflogy}
  If $Z$ is a single point then the fiber product $(X\times _{f,Z,g}
  Y, D(X\times_{f,Z,g} Y))$ is the product of $(X,D_X)$ and $(Y,D_Y)$.
  Thus the construction of the fiber product diffeology includes the
  construction of the (binary) product diffeology as a special case.
\end{remark}

Next we construct a diffeology on the space $\calP(M)$ of paths with
sitting instances in a manifold $M$ (q.v.\ Notation~\ref{calP(M)}).
By the quotient construction this, in turn, defines a diffeology on
the set $ \calP(M)/_{\sim}$ of paths modulo thin homotopy.

\begin{definition}[Path space diffeology]\label{def:pathspacedifeol}
  As before, denote the set of paths with sitting instances in a
  manifold $M$ by $\calP(M)$ (all paths in $\calP(M)$ are
  parameterized by $[0,1]$).  Let $U$ be an open set in some $\R^n$.
  Define a map of sets $p:U\to \calP(M)$ to be a plot if
  the associated map
\begin{equation}\label{phat}
\hat{p} : U\times [0,1] \to M, \qquad \hat{p}(u,x) := p(u)(x) 
\end{equation}
is smooth. \hfill $\Box$
\end{definition}
\begin{remark}
  Strictly speaking, one should check that the collection of plots on
  the set of paths $\calP(M)$ given in
  Definition~\ref{def:pathspacedifeol} forms a sheaf on $\Open$.
  This is straightforward, and we leave it to an interested reader.
\end{remark}
\begin{remark}
  If $X$ and $Y$ are two diffeological spaces, then the
  space $\Hom(X, Y)$ 
   of smooth maps from $X$ to $Y$ 
  is also a diffeological space: $p:U\to \Hom(X,Y)$ is a plot if
  and only if $\hat{p}:U\times X \to Y$ is smooth.  One can show that
  the mapping space diffeology on $\calP(M)$ agrees with the
  diffeology defined in Definition~\ref{def:pathspacedifeol}. The key
  issue is that a map $\hat{p}: U\times [0,1] \to M$ is smooth as a map
  of manifolds with boundary if and only if it is smooth as a map of
  diffeological spaces, see \cite{IZ}.
\end{remark}

\begin{lemma}\label{lem:conc} Let $\calP(M)$ be the set of paths with
  sitting instances in a manifold $M$ with path space diffeology
  (q.v.\ \ref{def:pathspacedifeol}).
\begin{enumerate} 
 \item The evaluation  maps
   \[ev_0:\calP(M)\to M, \;\;\; \gamma \mapsto \gamma(0)\quad \text{
     and } \quad ev_1:\calP(M)\to M, \;\;\; \gamma \mapsto \gamma(1)\]
   are smooth.

\item
The concatenation map
\[\tilde{\fm}:\calP(M)\times_{ev_0,M,ev_1}\calP(M) \to \calP(M)\]
defined by: 
\begin{equation} \label{eq:3.1concat}
\tilde{\fm}(\gamma,\tau)\, (t) := 
\left\{ 
\begin{array}{lrl} \tau (2t) & \mathrm{if} & t\in [0,1/2] \\
\gamma (2t-1) & \mathrm{if} & t \in [1/2,1] \\ 
\end{array} \right.  \qquad 
\end{equation}
 is a smooth.
\end{enumerate}
\end{lemma}
\begin{proof}
  By Definition \ref{def:pathspacedifeol} of the path space
  diffeology, for any plot $p:U\to \calP(M),$ we have $ev_0\circ
  p=\hat p|_{U\times \{0\}}$, where $\hat{p}: U\times [0,1] \to M$ is
  the associated map. Since $\hat{p}$ is smooth the map $ev_0$ is a
  smooth as well.  A similar argument shows that the map $ev_1$ is
  smooth.

Recall that a plot $p$  of the fiber product $\calP(M)\times_M\calP(M)$ is given by a pair of plots 
\[p=(p_1:U\to\calP(M),\ p_2:U\to\calP(M))\] 
with $p_1(x)(1)=p_2(x)(0)$ 
for all $x\in U$ (Construction~\ref{constr:fpdiffeo}). We have
\[
\widehat{\tilde{\fm}\circ p} \, (u,t) =
\left\{ 
\begin{array}{lrl} p_1(u,2t) & \mathrm{if} & t\in [0,1/2] \\
p_2(u, 2t-1) & \mathrm{if} & t \in [1/2,1] \\ 
\end{array} 
\right. 
\]
Since the map $\widehat{\tilde{\fm}\circ p}$ is smooth for all plots
$p$, the map $\tilde{\fm}\circ p$ is smooth for all plots $p$.  Hence
the map $\tilde\fm$ is smooth.
\end{proof}

\begin{definition}[Diffeology on the space of thin homotopy classes of paths]
  As before, denote the set of thin homotopy classes of paths with
  sitting instances in a manifold $M$ by $\calP(M)/_{\sim}$.  We
  define the diffeology on $\calP(M)/_{\sim}$ to be the quotient
  diffeology induced by the map $\calP(M)\to \calP(M)/_{\sim}$.
\end{definition}

\begin{definition}[q.v.\
  \protect{\cite[8.3]{IZ}}] \label{def:diffeolgroupoid} A {\sf
    diffeological groupoid} $\Gamma$ is a groupoid object in the
  category $\DS$ of diffeological spaces.  That is, the sets of
  objects and arrows, $\Gamma_0$ and $\Gamma_1,$ are diffeological
  spaces, and the structure maps $\fs,\ft,\fm,\fri,$ and $\fu$
  (q.v. Notation \ref{notation:gpds}) are maps of diffeological
  spaces.
\end{definition}

As observed in the Definition/Proposition~\ref{defprop:thingpd},
for a manifold $M$, the thin fundamental groupoid $\pithin (M)$ is a
groupoid with the set of objects the manifold $M$ and the set of
morphisms the set $ \calP(M)/_{\sim}$ of paths modulo thin
homotopy.  
We are now in position to state and
prove the following folklore result.

\begin{proposition}\label{thm:gpoidpf}
  The thin fundamental groupoid $\pithin(M)$ of a manifold $M$ is a
  diffeological groupoid.
\end{proposition}

\begin{proof}
  We need to show that the five structure maps $\fs, \ft, \fm, \fu,$ and $
  \fri$ of the groupoid $\pithin(M)=\{ \calP(M)/_{\sim} \}\toto M$ are
  maps of diffeological spaces, i.e. smooth
  (q.v. Notation~\ref{notation:gpds}).

We start with the source map $\fs$.  Recall that it is defined by 
\[
\fs [\gamma] = \gamma (0)
\]
for any class $[\gamma]\in \calP (M)/_{\sim}$.  Let $q:\calP(M)\to
\calP (M)/_{\sim}$ be the quotient map.  By Lemma~\ref{lemma:3.18}, it
is enough to show that for any plot $p:U\to \calP (M)$ the composite
map $f(u):= (\fs \circ q \circ p) (u) $ is smooth.  However, $\fs\circ
q$ is just the map $ev_0:\calP(M)\to M$ which is smooth by Lemma
~\ref{lem:conc}.  So $\fs$ is smooth.  A similar argument shows that
$\ft$ is smooth.

To show that the multiplication map $\fm$ is smooth, we
need to show that for any plot $p:U\to (\calP (M)/_{\sim}) \times _M
(\calP (M)/_{\sim})$ the composite map $\fm\circ p: U\to
\calP(M)/_{\sim}$ is a plot for the quotient diffeology on $(\calP
(M)/_{\sim})$.  Recall that $\fm: (\calP (M)/_{\sim} )\times _M (\calP
(M)/_{\sim}) \to \calP (M)/_{\sim}$ is defined by 
\[
\fm([\gamma], [\tau]) := [\tilde{\fm} (\gamma, \tau)]
\]
where $\tilde{\fm}: \calP (M)\times _M \calP(M)\to \calP (M)$ is the concatenation:
\begin{equation*} 
\tilde{\fm}(\gamma,\tau)\, (t) := 
\left\{ 
\begin{array}{lrl} \tau (2t) & \mathrm{if} & t\in [0,1/2] \\
\gamma (2t-1) & \mathrm{if} & t \in [1/2,1] \\ 
\end{array} \right.  \qquad 
\end{equation*}
(q.v.\ \eqref{eq:2.1}).  In other words, $\fm$ is defined to make the diagram
\begin{equation}\label{eq:3.2}
\xy
(-25,7)*+{\calP (M)\times _M \calP(M) }="1";
(25, 7)*+{\calP(M)}="2";
(-25,-7)*+{(\calP (M)/_{\sim} )\times _M (\calP(M)/_{\sim})}="3";
(25,-7)*+{\calP(M)/_{\sim}}="4";
 {\ar@{->}^{\qquad \tilde{\fm}} "1";"2"};
 {\ar@{->}^{q} "2";"4"};
 {\ar@{->}_{(q,q)} "1";"3"};
{\ar@{->}_{\qquad \fm} "3";"4"};
\endxy
\end{equation}
commute.  Recall once more that a plot $r:U\to \calP (M)/_{\sim}\times
_M \calP(M)/_{\sim} $ is of the form $r = (r_1, r_2)$ where $r_i:U\to
\calP(M)/_{\sim}$ are plots with $\fs \circ r_1 = \ft\circ r_2$.
Since both $r_1$ and $r_2$ locally lift with respect to the quotient
map $q:\calP(M)\to\calP(M)/_{\sim}$ to plots of $\calP(M)$, it follows
that for each point $x\in U$, we can find an open neighborhood $V$ of
$x$ in $U$ and plots $s_i:V\to \calP(M)$, such that $q\circ s_i =
r_i|_V$.  Furthermore, since $\fs \circ q=ev_0$ and $\ft \circ q
=ev_1$, the map $s=(s_1,s_2)$ is a plot for the fiber product diffeology on 
$\calP(M)\times_M\calP(M)$.

By Lemma~\ref{lem:conc} the concatenation $\tilde{\fm}$ is smooth.
Hence $\tilde{\fm}\circ s$ is a plot for $\calP(M)$ which descends to
the plot $q\circ\tilde{\fm}\circ s$ for $\calP(M)/_{\sim}$.
Since\eqref{eq:3.2} commutes, it follows that $\fm\circ r$ is locally
a plot for $\calP(M)$.  Since diffeologies are sheaves, $r$ is
globally a plot.  Therefore the map $\fm$ is smooth.  The proof that
the structure maps $\fri$ and $\fu$ are smooth is similar.
\end{proof}

We end the appendix with two technical results.

\begin{lemma}\label{lem:loc_sections}
  For any manifold $M$ the target map $\ft: \calP(M)/_{\sim}\to M$
  of the thin fundamental groupoid $\pithin(M)$ has local
  sections. More precisely, for any $[\gamma] \in \calP(M)/_{\sim}$,
  there is a neighborhood $U$ of $x= \gamma (1) = \ft([\gamma])$ and a
  smooth section $\sigma :U\to \calP(M)/_{\sim}$ of $\ft$ with $\sigma
  (x) = [\gamma]$.
\end{lemma}

\begin{proof}
  Choose a coordinate chart $\varphi:U\to \R^n$ ($n= \dim M$) with
  $x\in U$, $\varphi (x) = 0$ and $\varphi(U) = \R^n$.  The map  
\[
p:U\times [0,1]\to M, \qquad 
p(y,t) := \varphi\inv (\beta(t)\cdot\varphi (y))
\]
is smooth.  Here $\beta\in C^\infty ([0,1])$ is the function in
Remark~\ref{rmrk:5.3333}.  Then
\[
\check{p}:U\to \calP(M), \qquad \check{p}(y)\,(t):= \varphi\inv (\beta(t)\cdot\varphi (y))
\]
is smooth. By construction, $\check{p}(x)$ is the constant curve $1_x$.  Also
$\check{p}(y)(0) =x$ and $\check{p}(y)(1) = y$ for all $y\in U$.  Hence
\[
[\check{p}]:= q\circ p:U\to \calP(M)/_{\sim} 
\]
is a section of $\ft$ with $[\check{p}](x) = [1_x]$ and $\fs([\check{p}(y)]) = x$ for
all $y\in U$.  Since the multiplication $\fm:\calP(M)/_{\sim}
\times_M \calP(M)/_{\sim} \to \calP(M)/_{\sim} $ is smooth, the map
\[
\sigma:U\to \calP(M)/_{\sim}, \qquad \sigma(y):= [\check{p}(y)][\gamma]
\]
is smooth.  By construction, $\sigma$ is a desired section.
\end{proof}

\begin{proposition}\label{prop:3.27}
The assignment $M\mapsto \pithin (M)$ extends to a 
  functor 
\[
\pithin: \Man \to \Difgpd 
\]
from the category $\Man$ of manifolds
  to the category $ \Difgpd$  of diffeological groupoids.
\end{proposition}

\begin{proof}
 Let $f:M\to N$ be a smooth map between two manifolds.  We need to
  define a smooth functor $\pithin(f):\pithin(M)\to \pithin(N)$ so that on
  objects it is the map $f:M\to N$.

Define $f_*:\calP(M)\to \calP(N)$ by 
\[
f_*(\gamma) = f\circ \gamma
\]
for all paths $\gamma\in \calP(M)$.   If $p:U\to \calP(M)$ is a plot, then 
\[
((f_* \circ p)(u) )\,(t)  = f(p(u)(t)) = f(\hat{p}(u,t)),
\]
where, as before, $\hat{p}(u,t) = p(u)(t)$.  By definition of the
diffeology on the space of paths, the map $\hat{p}$ is smooth. Hence
$(u,t)\mapsto ((f_* \circ p)(u) )\,(t) $ is smooth.  It follows
\[
f_*\circ p:U\to \calP(N)
\]
is a plot.  Therefore $f_*$ is smooth. 
 Follow $f_*$ by the quotient
map $q_N:\calP(N)\to \calP(N)/_{\sim}$.  If $H:\gamma_0\Rightarrow
\gamma_1$ is a thin homotopy between two paths in $M$ then $f\circ H:
f_*(\gamma_0)\Rightarrow f_*(\gamma_1)$ is a thin homotopy between
their images in $\calP(N)$.  Hence $q_N\circ f_*$ induces a smooth map
$[f_*]:\calP(M)/_{\sim} \to \calP(N)/_{\sim}$ making the diagram
\begin{equation} \label{eq:489}
\xy
(-10,10)*+{\calP(M)}="1";
(20,10)*+{\calP(N)}="2";
(-10,-7)*+{\calP(M)/_{\sim}}="3";
(20,-7)*+{\calP(N)/_{\sim}}="4";
 {\ar@{->}^{f_*} "1";"2"};
 {\ar@{->}^{q_N} "2";"4"};
 {\ar@{->}_{q_M} "1";"3"};
{\ar@{->}_{[f_*]} "3";"4"};
\endxy
\end{equation}
commute.
Note that
\[
\fs ([f_*]([\gamma]))= 
\fs ([f\circ \gamma]) = f(\gamma(0)) = f(\fs([\gamma])).
\]
Similarly $\fs ([f_*]([\gamma]))= f(\ft([\gamma]))$.  Hence
\begin{equation} \label{eq:488}
\xy
(-10,10)*+{\calP(M)/_{\sim}}="1";
(20,10)*+{\calP(N)/_{\sim}}="2";
(-10,-7)*+{M\times M}="3";
(20,-7)*+{N\times N}="4";
 {\ar@{->}^{[f_*]} "1";"2"};
 {\ar@{->}^{(\fs, \ft)} "2";"4"};
 {\ar@{->}_{(\fs, \ft)} "1";"3"};
{\ar@{->}_{f} "3";"4"};
\endxy
\end{equation}
commutes. 
If $\gamma, \tau:[0,1]\to M$ are two paths
with $\gamma(0) = \tau(1)$ then $f\circ \gamma (0) = f\circ \tau (1)$.
Moreover the concatenation of $f\circ \gamma$ and $f\circ \tau$ is
the concatenation of $\gamma$ and $\tau$ followed by $f$:
\[
(f\circ \gamma) (f\circ \tau) = f\circ (\gamma\tau).
\]
It follows from the definition of multiplication in the thin
fundamental groupoid and the map $[f_*]$ that
\[
[f_*] ([\gamma][\tau]) = [f\circ \gamma][f\circ \tau].
\]
Thus $[f_*]$ preserves multiplication.  Since $f_*$ applied to a
constant path $1_x$ is the constant path $1_{f(x)}$, $[f_*]$ preserves
identities.  Define $\pithin(f)$  on objects to be
$f$ and on arrows $[f_*]$. Then $\pithin(f)$ is a functor.
It is smooth by construction.
\end{proof}

  \section{Functors between the cocycle categories}\label{app:2}
In this appendix we  prove\\

\noindent{\bf Proposition}~\ref{prop:last}
Suppose $\Gamma$ is a Lie groupoid, $\calX,\calY\to \Man$ are two
  categories fibered in groupoids and $F:\calX \to \calY$ is a
  1-morphism of fibered categories. Then the functor $F$ induces a
  functor
\[
F_\Gamma:\calX(\Gamma)\to \calY(\Gamma)
\]
between cocycle categories. Moreover if $F$ is an equivalence of
categories then so is $F_\Gamma$.\\

\begin{proof}
  Recall that objects and arrows of a CFG $\varpi: \calU\to \Man$ pull
  back along maps in $\Man$.  More precisely suppose $f:M\to N$ is a
  morphism in $\Man$ and $x$ is an object of $\calU(N)$.  Then there
  is an object $f^*x$ of $\calU(M)$ (``a pullback of $x$ by $f$'') and
  an arrow $\tilde{f}:f^*x \to x$ with $\varpi(\tilde{f}) = f$.  The
  pullback $f^*x$ is not unique, but any two pullbacks are isomorphic.
  If $x_1\xrightarrow{\xi} x_2$ is a morphism in $\calU(M)$ then given
  the choices of pullbacks $\tilde{f_i}: f^*x_i \to x_i$, $i=1,2$
  there is a unique arrow $f^*\xi: f^*x_1 \to f^*x_2$ making the
  diagram
\[
\xy
(0,6)*+{f^*x_1}="1"; 
(0,-6)*+{f^*x_2}="2";
(12,6)*+{x_1} ="3";(12,-6)*+{\xi_2} ="4";
{\ar@{->}^{\tilde{f}_1} "1";"3"};
{\ar@{->}_{\tilde{f}_2} "2";"4"};
{\ar@{->}^{\xi} "3";"4"};
{\ar@{->}_{f^*\xi} "1";"2"};
\endxy.
\]
commute.

At this point it is convenient to completely switch to simplicial
notion and set $d_0:\Gamma_1\to \Gamma_0$ to be the source map and
$d_1:\Gamma_1\to \Gamma_0$ to be the target map.

Suppose $(x,\varphi:d_0^*x\to d_1^*x)$ is an object of
$\calX(\Gamma)$. Then $F(x)$ is an object of $\calY(\Gamma_0)$.  Now
{\em choose}
\[
d_i^* F(x) := F(d_i^*x).
\]
Then $F(\varphi)$ is a morphism in $\calY(\Gamma_1)$ from $d_0^* F(x)$
to $d_1^* F(x)$.  We need to check that the pair $(F(x), F(\varphi))$
is an object of $\calY (\Gamma)$, i.e., that $F(\varphi)$ satisfies
the cocycle condition.  We set
\[
d_i^*d_j^*F(x) := F(d_i^*d_j^*x)\qquad i=0,1,2,\,j=0,1.
\]
By definition of an object in $\calX(\Gamma)$ the pullbacks 
\[
d_i^*\varphi: d_i^*d_0^* x \to d_i^*d_1^*x, \qquad i=0,1,2
\]
satisfy the cocycle equation
\[
d_2^* \varphi \, d_0^*\varphi = d_1^* \varphi.
\]
Since the diagrams
\[
\xy
(-12,10)*+{F(d_i^*d_0^*x)}="1"; 
(-12,-6)*+{F(d_i^*d_1^* x)}="2";
(12,10)*+{F(d_0^*x)} ="3";
(12,-6)*+{F(d_1^*x)} ="4";
{\ar@{->} "1";"3"};
{\ar@{->}^{F(\varphi)} "3";"4"};
{\ar@{->} "2";"4"};
{\ar@{->}_{F(d_i^*\varphi)} "1";"2"};
\endxy  .
\]
commutes in $\calY$ and since the horizontal arrows project to $d_i:\Gamma_1\times_{\Gamma_0}\Gamma_1 \to \Gamma_1$ ($i=0,1,2$) we have
\[
F(d_i^*\varphi) = d_i^* F(\varphi).
\]
Hence
\[
d_2^*F(\varphi)\,  d_0^*F(\varphi) =F(d_2^*\varphi)\, 
F(d_0^*\varphi) =F(d_2^*\varphi \, d_0^*\varphi) =F(d_1^*\varphi) =
d_1^*F(\varphi).
\]
We conclude that the pair $(F(x),F(\varphi))$ is an
object of $\calY(\Gamma)$.  Similarly if $\alpha:(x,\varphi) \to
(x',\varphi')$ is a morphism in $\calX(\Gamma)$ then
$F(\alpha):(F(x),F(\varphi)) \to (F(x'),F(\varphi'))$ is a morphism in
$\calY$.  It is routine to check that these maps on objects and
morphisms do assemble into a functor $F_\Gamma:\calX(\Gamma)\to
\calY(\Gamma)$.

If $F:\calX\to \calY$ is an equivalence of categories then it has a
weak inverse $H:\calY\to \calX$ so that $\gamma: H\circ F\Rightarrow
id_{\calX}$ and $\delta: F\circ H \Rightarrow id_\calY$ for some
natural isomorphism $\gamma$ and $\delta$.  These natural isomorphisms
then give rise to natural isomorphisms
\[
\gamma_\Gamma:H_\Gamma\circ F_\Gamma \Rightarrow id_{\calX(\Gamma)},
\]
\[
\delta_\Gamma:F_\Gamma\circ H_\Gamma \Rightarrow id_{\calY(\Gamma)}
\]
and the proposition follows.
\end{proof}

\end{appendix}


\begin{thebibliography}{WWWW}

\bibitem{BH} J.\ Baez and A.\ Hoffnung, 
Convenient categories of smooth spaces,
{\em Trans.\ Amer.\ Math.\ Soc.} {\bf 363} (2011), no.~11, 5789--5825.
 
\bibitem{B} J.W. Barrett, Holonomy and path structures in general
relativity and Yang-Mills Theory, {\em International Journal of
Physics} {\bf 30} (1991), 1171--1215.

\bibitem{BX} K. Behrend and P. Xu, Differentiable stacks and
gerbes,  {\em Journal of Symplectic Geometry} {\bf 9} (2011), 285--341.

\bibitem{BMW} I.\ Biswas, S.\ Majumder and M.L.\ Wong, Root stacks,
  principal bundles and connections. {\em Bulletin des Sciences
    Math{\'e}matiques}  {\bf 136} (2012), no.~4, 369--398.

\bibitem{BN} I.\ Biswas and F.\ Neumann, Atiyah sequences, connections
  and Chern-Weil theory for algebraic and differentiable stacks,
  \url{http://arxiv.org/abs/1311.4673}


\bibitem{Blo} C.\ Blohmann, Stacky Lie groups, {\em Int.\ Math.\ Res.\ Notices}, 2008, rnn082.

\bibitem{CP} A. Caetano and R.F. Picken, An axiomatic definition of
holonomy, {\em Int.\ J.\ of Math.} {\bf 5} (1994),
835--848.

\bibitem{Chen1} K-T. Chen, Iterated integrals of differential forms
  and loop space homology, {\em Annals of Math.} {\bf 97} (1973),
  217--246.

\bibitem{Chen2} K-T. Chen, Iterated integrals, fundamental groups and covering space, {\em Trans. Amer. Math. Soc.} {\bf 206} (1975), 83--98.



\bibitem{H} J. Heinloth, Some notes on differentiable stacks, {\em
    Mathematisches Institut, Seminars} (Y. Tschinkel ed.), (2004-05),
  Universit\"{a}t G\"{o}ttingen, 1--32.

\bibitem{IZ} P.\ Iglesias-Zemmour, 
   {\em Diffeology}, {Mathematical Surveys and Monographs} 
   {\bf 185}, {American Mathematical Society, Providence, RI},
   {2013},
   pp.\ xxiv+439.
  
 
 \bibitem{K} S.\ Kobayashi, La connexion des varietes fibrees I et
   II, {\em C.R.\ Acad.\ Sci.\ Paris.} {\bf 238} (1954), 318--319,
   443--444.

\bibitem{KN} S.\ Kobayashi and K.\ Nomizu, {\em Foundations of Differential Geometry I}, John Wiley \& Sons, Inc., New York, 1963.

\bibitem{GTX} C.\ Laurent-Gengoux, J.-L.\ Tu, and P.\ Xu, Chern-Weil
  map for principal bundles over groupoids, {\em Math.\
    Zeitschrift} {\bf 255} (2007), 451--491.

\bibitem{L} E.\ Lerman, Orbifolds as Stacks?. {\em L'Enseignment
Math\'{e}matique} {\bf 56} (2010), 315--363.


\bibitem{LaubingerDiss} M.\ Laubinger, {\em Differential geometry in
    cartesian closed categories of smooth spaces}, Ph.\ D disseration,
  Tulane University, 2008.


\bibitem{MacK} K.C.H.\ Mackenzie, {\em General Theory of Lie Groupoids
    and Lie Algebroids}, London Mathematical Society Lecture Note
  Series {\bf 213}, Cambridge University Press, Cambridge,
  2005. xxxviii+501 pp.

\bibitem{Metzler} D.\ Metzler, Topological and smooth stacks, 
  \url{http://arxiv.org/abs/math/0306176}(2003).


\bibitem{Noohi} B. Noohi, {\em Foundations of Topological Stacks I},
  \url{http://arxiv.org/abs/math/0503247}.


\bibitem{SW} U. Schreiber and K. Waldorf, Parallel transport and
functors, {\em Journal of Homotopy and Related Structures} {\bf 4} (2009),
187-244.

\bibitem{S} J.M. Souriau, Groupes differentiels, in Lecture Notes in
Math., {\bf 836} (1980), Springer , pp. 91--128.

\bibitem{V} A. Vistoli, Grothendieck topologies, fibered categories,
and descent theory, in  {\em Fundamental algebraic geometry}, 1--104, {\em
Math Surveys Monograph}, {\bf 123}, American Mathematics Society,
Providence, Rhode Island, 2005.



\bibitem{Wood} E.E.\ Wood, Reconstruction Theorem for Groupoids and
  Principal Fiber Bundles, {\em International J.\ Theoret.\
    Phys.}, {\bf 36}, No.\ 5, 1997.
  
\end{thebibliography}
\end{document}